\documentclass{siamart1116}
\usepackage{amsmath}
\usepackage{amssymb}
\usepackage{graphicx}
\usepackage{xcolor}
\usepackage{hyperref}
\usepackage{bm}
\usepackage{cleveref}
\usepackage{enumerate}

\newcommand{\Mb}[1]{\left[{#1}\right]}

\newcommand{\Mcb}[1]{\left\{{#1}\right\}}

\newcommand{\vvx}{\vec{\bm{x}}}
\newcommand{\vvn}{\vec{\bm{n}}}
\newcommand{\vvy}{\vec{\bm{y}}}
\newcommand{\vvz}{\vec{\bm{z}}}
\newcommand{\vvxi}{\vec{\bm{\xi}}}
\newcommand{\vvzero}{\vec{\bm{0}}}

\newcommand{\vx}{\bm{x}}

\newcommand{\vz}{\bm{z}}
\newcommand{\vxi}{\bm{\xi}}

\renewcommand{\Im}{\mathrm{Im}}
\newcommand{\real}{\mathbb{R}}

\renewcommand{\hat}{\widehat}

\DeclareMathOperator{\supp}{supp}
\DeclareMathOperator{\linspan}{span}

\newcommand{\cB}{\mathcal{B}}
\newcommand{\cF}{\mathcal{F}}

\crefname{subsection}{section}{sections}
\Crefname{subsection}{Section}{Sections}

\title{Approximation by Herglotz wave functions%
 \thanks{\funding{Army Research Office Contract No.
 W911NF-16-1-0457.}}}

\author{Fernando Guevara Vasquez%
\thanks{Mathematics Department, University of Utah, 155 S 1400 E RM
233, Salt Lake City UT 84112-0090 (\email{fguevara@math.utah.edu},
\email{mauck@math.utah.edu}).}
\and
  China Mauck%
  \footnotemark[2]
}

\headers{Approximation by Herglotz wave functions}{F. Guevara Vasquez
and C. Mauck}

\begin{document}
\maketitle

 \begin{abstract}
 We consider the problem of approximating a function using Herglotz wave
 functions, which are a superposition of plane waves. When the
 discrepancy is measured in a ball, we show that the problem can
 essentially be solved by considering the function we wish to
 approximate as a source distribution and time reversing the resulting
 field. Unfortunately this gives generally poor approximations.
 Intuitively, this is because Herglotz wave functions are determined by
 a two-dimensional field and the function to approximate is
 three-dimensional. If the discrepancy is measured on a plane, we
 show that the best approximation corresponds to a low-pass filter,
 where only the spatial frequencies with length less than the wavenumber
 are kept. The corresponding Herglotz wave density can be found
 explicitly.  Our results have application to designing standing
 acoustic waves for self-assembly of micro-particles in a fluid.
 \end{abstract}



\section{Introduction}
\label{sec:intro}

We study the problem of finding the best approximation of a function by
{\em Herglotz wave functions}, which are functions of the form
\begin{equation}
 \label{eq:hwf}
 u(\vvx) = \int_{S(0,k)} g(\vvx) \exp[i \vvx \cdot \vvz] dS(\vvz),
\end{equation}
where the integral is over the sphere $S(0,k) \equiv \{ \vvx \in \real^3
~|~ |\vvx| = k\}$, $k = 2\pi/\lambda$ is the wavenumber corresponding to
a wavelength $\lambda$ and we refer to  $g(\vvx)$ as a {\em Herglotz wave
density}\footnote{Herglotz wave functions may also be defined as
integrals over $S(0,1)$, see e.g. \cite[\S 3]{Colton:1998:IAE}.}.
Herglotz wave functions are entire solutions to the Helmholtz equation
$\Delta u + k^2 u = 0$. 

The application we have in mind for this
approximation problem is ultrasound directed self-assembly of
micro-particles
\cite{Greenhall:2013:CUM,Greenhall:2015:UDS,Prisbrey:2017:UDS}, where
micro-particles in a fluid are controlled with standing acoustic waves.
For this application, the fluid pressure $u$ inside a reservoir $D$ (an
open subset of $\real^3$) with Lipschitz boundary $\partial D$ satisfies
 \begin{equation}
  \label{eq:reservoir}
  \begin{aligned}
   \Delta u + k^2 u &= 0~\text{in}~D,~\text{and}\\
   \vvn \cdot \nabla u + i k \alpha u &= \phi~\text{on}~\partial D,
  \end{aligned}
 \end{equation}
where $\vvn(\vvx)$ is the unit outward pointing normal vector to the
boundary and $\alpha(\vvx)$ is a function representing the boundary
impedance of $\partial D$. The reservoir boundary $\partial D$ is assumed to be lined with
transducers and the external excitation $\phi(\vvx)$  models the transducer
operating parameters, i.e. the amplitude and phase of the voltage driving the
transducers. The problem is to find transducer
operating parameters $\phi$ such that particles cluster in a desired pattern,
e.g. a surface or a curve. If the particles are neutrally
buoyant in the fluid and less compressible than the fluid, the particles are known to
cluster about the nodal or zero level set of the pressure, i.e. $\{ \vvx
\in D ~|~ u(\vvx) = 0 \}$ (see e.g.  \cite{Settnes:2012:FAS}). The
strategy we propose is to construct a function
$f$ defined on $D$, with nodal set containing the desired pattern.
If $v$ is a Herglotz wave function that is close enough to $f$, we
expect their nodal sets to be close as well. Thus taking $\phi = \vvn
\cdot \nabla v + i k \alpha v$ gives a function $v$ solving
\eqref{eq:reservoir} whose nodal set is close to the desired pattern.

\subsection{Related work}
The forces that a standing acoustic field in a fluid exerts on
compressible particles can be described by an acoustic radiation
potential \cite{King:1934:ARP, Yosioka:1955:ARP, Gorkov:1962:FAS,
Settnes:2012:FAS}, whose minima correspond to locations where the
particles tend to cluster. The results we present here apply only to the
case where the minima coincide with the zero level set of the wave
field, e.g. when the particles are less compressible than the fluid and are
neutrally buoyant in the fluid. The problem of designing Helmholtz equation
solutions for which the minima of the acoustic radiation potential are
close to a desired pattern has been studied numerically and
experimentally in both 2D and 3D settings
\cite{Greenhall:2015:UDS,Prisbrey:2017:UDS}. The numerical approach in
\cite{Greenhall:2015:UDS,Prisbrey:2017:UDS} allows for a broader range
of particle and fluid parameters than the one we consider here and also
accounts for having finitely many transducers lining the reservoir.  The
design problem is formulated as a constrained minimization problem in
\cite{Greenhall:2015:UDS}. The objective function is a quadratic
functional representing an aggregate of the acoustic radiation potential
at the points where particles are desired.  A quadratic constraint
restricts the minimization to transducer excitations with the same total
power. The minimization can be solved efficiently as it is equivalent to
finding the eigenvector corresponding to the smallest (algebraically
speaking) eigenvalue of a matrix \cite{Greenhall:2015:UDS}.

We emphasize that we are interested in approximating functions that are
not necessarily Helmholtz equation solutions with Herglotz wave
functions. Thus, the problem we consider here is fundamentally different
from the results showing that Herglotz wave functions are dense in the
space of solutions to the Helmholtz equation \cite{Weck:2004:AHW}.
Another related problem is that of approximating an entire solution to the
Helmholtz equation in a region by a linear combination of singular
Helmholtz equation solutions. This can be done using Green's identities
(see e.g.  \cite{Colton:1998:IAE}) and has applications to active
cloaking as proposed in
\cite{Miller:2007:PC}. Indeed, such an
approximation scheme could be used to significantly reduce an incident
field within a region, which in turn suppresses scattering from any
object that we wish to hide within the region. Other approaches to solve
the same problem, while not completely surrounding the object with
sources, include \cite{Guevara:2011:ECA,Onofrei:2014:AMF}.

\subsection{Contents}
We first consider the approximation problem on a ball of radius $R$ in
\cref{sec:herglotz:vol}. The solution to this problem is related to
a time reversal experiment where the source density is the function we
wish to approximate.  When the approximation problem is restricted to a plane
(\cref{sec:herglotz:plane}), the best approximation is not related to
time reversal but is a low pass filtered version of the function we wish
to approximate. We illustrate both approaches with numerical experiments in
\cref{sec:numerics} and conclude with a summary and future work in
\cref{sec:future}.

\section{Approximation by Herglotz wave functions restricted to a ball}
\label{sec:herglotz:vol}
In \cref{sec:heuristic} we give a heuristic that motivates the
best approximation result we are after. Some facts about time reversal
are recalled in \cref{sec:timerev} and its connection with the heuristic
is in \cref{sec:conv}. For the best approximation result, we
work on a space of Herglotz wave functions restricted to the ball
$B(0,R)$, which is defined and studied in \cref{sec:restricted}. The
projection of a function into this space is carried out in
\cref{sec:proj}. Finally we explain in \cref{sec:relation} how the
projection and the heuristic are related.
\subsection{A motivating heuristic}
\label{sec:heuristic}
Let $u$ be a solution to the Helmholtz equation with wavenumber $k$.
If the Fourier transform $\hat u(\vvxi)$ of $u$ is well-defined (perhaps
in the sense of distributions), it must satisfy
\begin{equation}
 \label{eq:symbol}
 (-|\vvxi|^2 + k^2) \hat u (\vvxi) = 0.
\end{equation}
In particular we must have $\supp \hat u \subset S(0,k)$. Thus if we are
given a function $f$ to approximate with a Helmholtz equation solution,
it would make sense to use as approximation a function $g$ whose Fourier
transform $\hat g(\vvxi)$ coincides with $\hat f(\vvxi)$ on the sphere
$|\vvxi| = k$, but is zero elsewhere. This heuristic of ``filtering out
everything outside of $S(0,k)$ in spatial frequency'' sets the
expectations for the approximation result we are after. First we need to
use appropriate spaces to be able to compare a function defined on some
subset of $\real^3$ with functions defined on a sphere. Second, the
approximation can be very poor, since a function with zero Fourier
transform on the sphere $|\vvxi| = k$ but non-zero elsewhere would be
approximated by the zero solution to the Helmholtz equation.

\subsection{Time reversal and convolution with a spherical Bessel
function}
\label{sec:timerev}
{\em Time reversal} consists in recording an acoustic field at a surface
(the time reversal mirror), time reversing it and propagating it back in
the medium \cite{Fink:1997:TRA}. The field at location $\vvx$ generated
by a point source in a homogeneous medium at location
$\vvy$ is $G(\vvx-\vvy,k)$, where $G$ is the free space Green function
for the Helmholtz equation
\begin{equation}
 G(\vvx,k) = \frac{\exp[i k |\vvx|]}{4\pi |\vvx|}.
 \label{eq:green}
\end{equation}
The time reversal procedure applied to the field $G(\vvx-\vvy,k)$, with
time reversal mirror being the sphere $S(0,R)$, gives
the field
\begin{equation}
 u_{\vvy} (\vvx) = \int_{S(0,R)} G(\vvx-\vvz,k) \overline{G(\vvz-\vvy,k)}
 dS(\vvz)  = \frac{1}{k} \Im~G(\vvx-\vvy, k),
\end{equation}
assuming the source is inside the sphere, i.e. $|\vvy| < R$.
This follows from the Helmholtz-Kirchhoff identity
\cite[\S 2.1]{Garnier:2016:PIA}. Thus the imaginary part of the Green
function gives the tightest possible spot that can be focused using
waves in a homogeneous medium at wavenumber $k$. Using \eqref{eq:green}
and the zero-th order spherical Bessel function $j_0(t) = \sin(t) / t$
we get that 
\[
 u_{\vvy} (\vvx) = (4\pi)^{-1} j_0(k|\vvx-\vvy|).
\]
The same principle can be applied to more complicated source
densities. If $f$ is an $L^\infty$ real valued source density function with
$\supp f \subset B(0,R)$, the resulting field is $G(\cdot,k) * f$ and time
reversing it with a time reversal mirror on $S(0,R)$ gives
$(4\pi)^{-1}j_0(k|\cdot|) * f$. Hence time reversal and convolution with
$(4\pi)^{-1}j_0(k|\cdot|)$ are equivalent in this setting.

\subsection{Convolution with a spherical Bessel function and the heuristic}
\label{sec:conv}
As we see in the next lemma, convolution with $j_0(k|\cdot|)$ (and thus
time reversal) is equivalent (up to a multiplicative constant) to the heuristic of ``filtering
everything outside of $S(0,k)$ in spatial frequency''.
\begin{lemma}
Let $f$ be a compactly supported $L^\infty$ function, then 
 \begin{equation}
  [f*j_0(k|\cdot|)](\vvx) = \frac{1}{4\pi k^2} \int_{S(0,k)} \hat
  f(\vvz) \exp[i \vvx \cdot \vvz] dS(\vvz).
 \end{equation}
In other words, $f*j_0(k|\cdot|)$ is a Herglotz wave function with wavenumber
 $k$ and density $\hat f(\vvz) |_{S(0,k)} / (4\pi k^2)$.
\end{lemma}

\begin{proof}
 By the Funk-Hecke formula (see e.g. \cite[\S 2.4]{Colton:1998:IAE}), we have
 that
 \begin{equation}
  j_0(k|\vvx|) = \frac{1}{4\pi k^2} \int_{S(0,k)} \exp[i \vvx \cdot
  \vvz] dS(\vvz).
 \end{equation}
 Using the previous expression for $j_0(k|\cdot|)$ in the convolution
 gives the desired result:
 \[
 \begin{aligned}
  (f*j_0(k|\cdot|)) (\vvx) &= \frac{1}{4\pi k^2} \int d\vvy \, f(\vvy)
  \int_{S(0,k)}
  dS(\vvz) \exp[i (\vvx-\vvy) \cdot \vvz ] \\
 &=\frac{1}{4\pi k^2} \int_{S(0,k)} dS(\vvz) \, \exp[ i \vvx \cdot \vvz ]
 \int d\vvy \, \exp[-i \vvy \cdot \vvz] f(\vvy)\\
 &= \frac{1}{4\pi k^2} \int_{S(0,k)} dS(\vvz) \, \exp[i \vvx \cdot \vvz ]
 \hat f (\vvz).
 \end{aligned}
 \]
 We point out that $\hat f$ is $C^\infty$ by the Paley-Wiener theorem
 since $f$ is compactly supported.
\end{proof}

\subsection{Herglotz wave functions restricted to a ball}
\label{sec:restricted}
Let us consider the family of spherical wave
functions
\begin{equation}
 u_{nm}(\vvx) =  j_n(k|\vvx|) Y_{nm}(\vvx/|\vvx|), n = 0,1,\ldots, m
 =-n,\ldots n,
\end{equation}
where the $Y_{nm}$ are spherical harmonics defined as in
\cite{Colton:1998:IAE} and normalized to be orthonormal in
$L^2(S(0,1))$.
The $u_{nm}$ are solutions to the Helmholtz equation and by the
Funk-Hecke formula  they are also Herglotz
wave functions \cite[\S 2.4]{Colton:1998:IAE}. Indeed we have for  $n = 0,1,\ldots, m
 =-n,\ldots n$ that
\begin{equation}
 u_{nm}(\vvx) = \frac{i^n}{4\pi k^2} \int_{S(0,k)} \exp[ i \vvx \cdot
 \vvz] Y_{nm}(\vvz/k)dS(\vvz).
 \label{eq:fh}
\end{equation}
We consider the $(N+1)^2$ dimensional space of Herglotz wave functions
\begin{equation}
 S_{R,N} \equiv \linspan \Mcb{ u_{nm}|_{B(0,R)},~n = 0,\ldots,N,~m
 =-n,\ldots n }.
\end{equation}
The orthogonality of the spherical harmonics $Y_{nm}$ with respect to the $L^2(S(0,1))$
inner product guarantees that the spherical wave functions $u_{nm}$ (with
order up to $N$) form an orthogonal basis of $S_{R,N}$ with respect to
the $L^2(B(0,R))$ inner product. The next lemma shows that we can, for
all practical purposes, think of the functions $(k\sqrt{2}/R)u_{nm}$ as
an {\em orthonormal} basis for $S_{R,N}$ provided $R$ is sufficiently
large.
\begin{lemma}
 \label{lem:unm}
 As $R \to \infty$ and for $N$ fixed, the spherical wave functions for $n=0,\ldots,N$ and
 $m=-n,\ldots,n$ satisfy
 \[
  \frac{2 k^2}{R} \| u_{nm} \|_{L^2(B(0,R))}^2 \to 1.
\]
\end{lemma}
\begin{proof}
The norm of a spherical wave function $u_{nm}$ restricted to the ball $B(0,R)$
is
\[
\| u_{nm} \|_{L^2(B(0,R))}^2 =  \int_0^R r^2 [j_n(kr)]^2 \, dr
=  \frac{1}{k^3} \int_0^{kR} t^2  [j_n(t)]^2  \,dt.
\]
Using the asymptotic (see e.g. \cite[\S 3.3]{Colton:1998:IAE})
\[
    \lim_{T \to \infty} \frac{1}{T} \int_0^T r^2 [j_n(r)]^2 dr =
    \frac{1}{2},
\]
we see that $ R^{-1}\| u_{nm} \|_{L^2(B(0,R))}^2 \to (2k^2)^{-1}$ as $R \to
\infty$ and $N$ is kept fixed. 
\end{proof}

\subsection{Projection onto a space of restricted Herglotz wave functions}
\label{sec:proj}
Let $f$ be an $L^\infty$ function with compact support $D$ and assume
$R$ is large enough so that $D \subset B(0,R)$. Then clearly $f \in
L^2(B(0,R))$ and the best approximation $f_{R,N}$ of $f$ by functions in
$S_{R,N}$ is given by its orthogonal projection:
\begin{equation}
   f_{R,N} = \sum_{n=0}^N \sum_{m=-n}^n \frac{\langle u_{nm},
   f\rangle}{\langle u_{nm}, u_{nm} \rangle} u_{nm},
   \label{eq:proj}
\end{equation}
where $\langle \cdot , \cdot \rangle$ is the inner product on
$L^2(B(0,R))$. The inner products $\langle u_{nm}, f\rangle$ measure
spherical harmonic coefficients of $\hat f$ on the sphere $S(0,k)$, as
we show next.
\begin{lemma}
\label{lem:innerprods}
Let $f$ be an $L^\infty$ function with $\supp f \subset B(0,R)$. Then we have
\[
 \langle u_{nm}, f \rangle_{L^2(B(0,R))} = (4\pi i^n k^2)^{-1} \langle Y_{nm}(\cdot/k) , \hat f
\rangle_{L^2(S(0,k))}.
\]
\end{lemma}
\begin{proof}
Writing the $L^2(B(0,R))$ inner product and using the Funk-Hecke formula
\eqref{eq:fh} we get
 \[
   \begin{aligned}
    \langle u_{nm}, f \rangle_{L^2(B(0,R))} 
    &= \frac{1}{4\pi k^2 i^n} \int_{B(0,R)} d\vvx \, 
    f(\vvx)
    \left[  \int_{S(0,k)} dS(\vvz) \exp[ -i \vvx \cdot \vvz]
    \overline{Y}_{nm}(\vvz/k) \right]\\
    &= \frac{1}{4\pi k^2 i^n} \int_{S(0,k)} dS(\vvz)\, \overline{Y}_{nm}(
    \vvz/k) \int_{B(0,R)} d\vvx f(\vvx) \exp[-i \vvx \cdot \vvz]\\
    &= \frac{1}{4\pi k^2 i^n} \int_{S(0,k)} dS(\vvz)\,
    \overline{Y}_{nm}(\vvz/k)
    \hat f(\vvz),
   \end{aligned}
  \]
  which proves the desired result. 
\end{proof}
The best approximation $f_{R,N}$ is a Herglotz wave function on
$S(0,k)$. Its density is given in the next lemma.
\begin{lemma}
 \label{lem:density}
 The best approximation $f_{R,N}$ in \eqref{eq:proj} can be written as a Herglotz wave
 function on $S(0,k)$ with density
 \begin{equation}
  g_{R,N}(\vvz) = \frac{1}{(4\pi k^2)^2}\sum_{n=0}^N \sum_{m=-n}^n  \frac{\langle Y_{nm}(\cdot/k),\hat
   f\rangle_{L^2(S(0,k))}}{\langle u_{nm}, u_{nm}
   \rangle_{L^2(B(0,R))}} Y_{nm}(\vvz/k).
  \label{eq:frn:density}
 \end{equation}
\end{lemma}
\begin{proof}
 In the definition \eqref{eq:proj} of $f_{R,N}$ we use the Funk-Hecke
 formula \eqref{eq:fh} to replace the right-most $u_{nm}$ and get 
 \[
 \begin{aligned}
  f_{R,N}(\vvx) &= \sum_{n=0}^N \sum_{m=-n}^n \frac{\langle u_{nm},
   f\rangle}{\langle u_{nm}, u_{nm} \rangle} \frac{i^n}{4\pi k^2}\int_{S(0,k)} \exp[i\vvz \cdot
  \vvx] Y_{nm}(\vvz/k) dS(\vvz)\\
  &= \int_{S(0,k)}  \exp[i\vvz \cdot
  \vvx] g_{R,N}(\vvz) dS(\vvz),
 \end{aligned}
\]
where $g_{R,N}(\vvz)$ is given for $\vvz \in S(0,k)$ by
\[
 \begin{aligned}
 g_{R,N}(\vvz) &= \sum_{n=0}^N \sum_{m=-n}^n \frac{\langle u_{nm},
   f\rangle}{\langle u_{nm}, u_{nm} \rangle} \frac{i^n}{4\pi k^2}
   Y_{nm}(\vvz/k)\\
   &=\frac{1}{(4\pi k^2)^2}\sum_{n=0}^N \sum_{m=-n}^n  \frac{\langle Y_{nm}(\cdot/k),\hat
   f\rangle_{L^2(S(0,k))}}{\langle u_{nm}, u_{nm}
   \rangle_{L^2(B(0,R))}} Y_{nm}(\vvz/k),
 \end{aligned}
\]
where we used \cref{lem:innerprods} for the second equality.
\end{proof}

\subsection{Relation between projection and heuristic}
\label{sec:relation}
We now use the $R$ large asymptotic result for $\|u_{nm}\|_{L^2(B(0,R))}$ to show that
the Herglotz density of the best approximation $f_{R,N}$ approaches that
of 
the projection of $\hat f$ onto the spherical harmonics up to order $N$
in $S(0,k)$. This shows that the heuristic of filtering out everything
outside of the sphere $S(0,k)$ in spatial frequency is related to
approximating a function by Herglotz wave functions. Note that we need
to work in $L^2(B(0,R))$ because entire solutions to the Helmholtz
equation are not in $L^2(B(0,R))$ (because of their growth at infinity,
see e.g. \cite[\S 3.3]{Colton:1998:IAE}).
\begin{theorem}
\label{thm:relation}
Let $g_{R,N}$ be the Herglotz density on $S(0,k)$ defined in
\eqref{eq:frn:density}. Then for a fixed $N$ we have as $R \to \infty$
\[
 8\pi^2 R g_{R,N} \to  \sum_{n=0}^N \sum_{m=-n}^n \langle
 k^{-1}Y_{nm}(\cdot/k), \hat f \rangle_{L^2(S(0,k))} k^{-1}Y_{nm}(\cdot/k),
\]
where the convergence is understood in $L^2(S(0,k))$ and the limiting
function is the projection of $\hat f$ onto the spherical harmonic basis
of $S(0,k)$ up to order $N$.
\end{theorem}
\begin{proof}
We recall that the functions $k^{-1}Y_{nm}(\cdot/k)$ form an orthonormal
basis of $L^2(S(0,k))$.  By \cref{lem:density}, the coefficient of $
8\pi^2 R g_{R,N}$ along the basis function $k^{-1} Y_{nm}(\cdot/k)$ is
\[
 \langle  k^{-1} Y_{nm}(\cdot/k), 8\pi^2 R g_{R,N}\rangle_{L^2(S(0,k))} =
 \frac{R}{2k^2} \frac{\langle k^{-1} Y_{nm}(\cdot/k),\hat
   f\rangle_{L^2(S(0,k))}}{\langle u_{nm}, u_{nm}
   \rangle_{L^2(B(0,R))}}.
\]
Using \cref{lem:unm} we get that as $R \to \infty$,
\[
 \langle  k^{-1} Y_{nm}(\cdot/k),8\pi^2 R g_{R,N} \rangle_{L^2(S(0,k))}
 \to \langle k^{-1}Y_{nm}(\cdot/k), \hat f \rangle_{L^2(S(0,k))}.
\]
This shows the desired result. 
\end{proof}

\section{Approximation by Herglotz wave functions restricted to a plane}
\label{sec:herglotz:plane}
We would like to find a Herglotz wave function whose restriction to a
plane is as close as possible to a function defined on the same plane.
To study this problem we introduce a weighted space of Herglotz wave
functions (\cref{sec:weighted}) and also a space of band-limited
functions (\cref{sec:bandlimited}). We establish a one-to-one
correspondence between these spaces and use this fact to find the best
approximation of a function in $L^2(\real^2)$ by a Herglotz wave
function restricted to a plane (\cref{sec:bestapprox}). Since time
reversal can be used to express the solution to the best approximation
by Herglotz wave functions in a volume (\cref{sec:herglotz:vol}), we
apply the same principle to approximate a function defined on a plane
and show that the time reversal solution is suboptimal (\cref{sec:tr}).
The final result is that the approximation problem we consider is (up to
multiplicative constants) equivalent to filtering out all the spatial
frequencies $\vxi$ such that $|\vxi| > k$ in the function we want
to approximate. This limits the resolution that is achievable by this
approach to about a wavelength, which is consistent with the Rayleigh
resolution limit for imaging with waves (see e.g.
\cite{Bleistein:2013:MSI}). To see this, consider the distribution
$\delta(\vx)$. Filtering out all the spatial frequencies of
$\delta(\vx)$ outside of $B(0,k)$ gives the function
\[
 (\chi_{B(0,k)})^\vee(\vx) =  (2\pi)^{-1} \frac{J_1(k|\vx|)}{|\vx|},
\]
where $\vee$ denotes the inverse Fourier transform.  The first zero of
$J_1(kz)/z$ is located at $z\approx 1.22\lambda/2$ (see e.g. \cite[\S
10]{dlmf}), thus the smallest feature we can resolve is about
1.22$\lambda$. Hereinafter, vectors $\vvx \in \real^3$ have
arrows and the first two components of $\vvx \in \real^3$ are denoted by
$\vx \in \real^2$, i.e.  we have $\vvx = (\vx,x_3)$.

\subsection{A weighted space of Herglotz wave functions}
\label{sec:weighted}
For a given wavenumber $k$, let us define the space $S_k$ of Herglotz
wave functions with density in a weighted $L^2$ space on the sphere of
radius $k$:
\begin{equation}
 \label{eq:sk}
 S_k = \Mcb{ \int_{S(0,k)} g(\vvz) e^{i\vvz \cdot \vvx} dS(\vvz)
 ~|~ g \in L^2_W(S(0,k)) },
\end{equation}
where $g \in L^2_W(S(0,k))$ if and only if 
\[
\int_{S(0,k)} \frac{|g(\vvz)|^2}{|z_3|} dS(\vvz) < \infty.
\]
Naturally, we need to make sure that the integral in the definition of a
function in $S_k$ is well defined. This follows from the Cauchy-Schwartz
inequality:
\[ 
 \begin{aligned}
  \int_{S(0,k)} g(\vvz) e^{i\vvz \cdot \vvx} dS(\vvz) &= \int_{S(0,k)}
  \frac{g(\vvz)}{|z_3|^{1/2}} |z_3|^{1/2} e^{i\vvz \cdot \vvx} dS(\vvz)\\
  &\leq \Mb{\int_{S(0,k)} \frac{|g(\vvz)|^2}{|z_3|} dS(\vvz) }^{1/2}
  \Mb{\int_{S(0,k)} |z_3| dS(\vvz)}^{1/2},
 \end{aligned}
\]
where the upper bound is finite when $g \in L^2_W(S(0,k))$.

\subsection{Relation between Herglotz wave functions and band-limited
functions}
\label{sec:bandlimited}
For the wavenumber $k$, we define the space of
band-limited functions 
\begin{equation}
 \label{eq:bk}
 B_k = \{ f \in L^2(\real^2) ~|~ \supp \hat f \subset B(0,k) \},
\end{equation}
where $B(0,k)$ is the ball of radius $k$ centered at the origin.
In the next lemma, we
show that $B_k$ can be identified with the space of restrictions of
elements of $S_k$ to the plane $x_3 = 0$. In fact we only need to
consider the subset $S_{k+}$ of $S_k$ consisting of Herglotz wave
functions with density $g \in L^2_W(S(0,k))$, supported on the half sphere
$\{ \vvz ~|~ \vvz\in S(0,k), z_3 \geq 0\}$.
\begin{lemma}
 \label{lem:herglotz}
 Let $u \in S_{k+}$, then the function $f : \vx \to u(\vx,0)$ is a function
 in $B_k$.  Conversely for any function $f \in B_k$, there is a Herglotz
 wave function $u \in S_{k+}$ such that $f(\vx) = u(\vx,0)$.
\end{lemma}
\begin{proof}
Let $u \in S_{k+}$ and let $g$ be its density.  Restricting $u$ to the
plane $x_3 = 0$, we get
\[
\begin{aligned}
u(\vx,0) &= \int_{S(0,k),z_3\ge 0} g(\vvz) e^{i(\vz,z_3)\cdot (\vx,0)}
dS(\vvz)\\
&= \int_{B(0,k)} g(\vz,\sqrt{k^2 - |\vz|^2}) e^{i\vz\cdot
\vx}\frac{k}{\sqrt{k^2-|\vz|^2}}d\vz.
\end{aligned}
\]
Therefore $u(\vx,0)$ is the Fourier transform of the $L^2(B(0,k))$ function
\begin{equation}
 v(\vz) = (2\pi)^2 g(\vz,\sqrt{k^2 - |\vz|^2}) \frac{k}{\sqrt{k^2-|\vz|^2}}\chi_{B(0,k)}(\vz),
\end{equation}
and $u(\vx,0) \in B_k$. Indeed we have
\[
\begin{aligned}
\int_{B(0,k)} |v(\vz)|^2 d\vz &= (2\pi)^4 k^2 \int_{B(0,k)} \bigg| \frac{g(\vz, \sqrt{k^2 - |\vz|^2})}{\sqrt{k^2 - |\vz|^2}} \bigg|^2 d\vz \\
&= (2\pi)^4 k^2 \int_{S(0,k), z_3\ge 0} \frac{|g(\vvz)|^2}{z_3^2} \frac{\sqrt{k^2-|\vz|^2}}{k} dS(\vvz) \\
&= (2\pi)^4 k \int_{S(0,k), z_3\ge 0} \frac{|g(\vvz)|^2}{z_3} dS(\vvz) <
\infty. 
\end{aligned}
\]
Now take a function $f\in B_k$. Since $\hat f$ is supported in $B(0,k)$ we have
\begin{align*}
f(\vx) &= (2\pi)^{-2} \int_{B(0,k)} \hat{f}(\vz)e^{i\vz\cdot \vx} d\vz \\
&= (2\pi)^{-2} \int_{S(0,k),z_3\ge 0} \hat{f}(\vz) e^{i(\vz,z_3)\cdot
(\vx,0)}\frac{\sqrt{k^2-|\vz|^2}}{k} dS(\vvz).
\end{align*}
Thus $f(\vx)$ is the restriction to the plane $x_3=0$ of a Herglotz wave
function with density
\[
 g(\vvz) = (2\pi)^{-2} \hat{f}(\vz) \frac{\sqrt{k^2-|\vz|^2}}{k}
 \chi_{z_3 \geq 0}(\vvz) = (2\pi)^{-2} \hat{f}(\vz) \frac{(z_3)_+}{k},
\]
where $x_+ \equiv  (x + |x|)/2$. Notice $g \in L^2_W(S(0,k))$ because
\[
\begin{aligned}
\int_{S(0,k)} \frac{|g(\vvz)|^2}{|z_3|} dS(\vvz) &=
\frac{(2\pi)^{-4}}{k^2}  \int_{S(0,k), z_3\ge 0} \frac{|\hat{f}(\vz)
(z_3)_+|^2}{|z_3|} dS(\vvz) \\
&= \frac{(2\pi)^{-4}}{k} \int_{B(0,k)} |\hat{f}(\vz)|^2 d\vz < \infty.
\end{aligned}
\]
\end{proof}

\subsection{Best approximation by Herglotz wave functions restricted to
a plane}
\label{sec:bestapprox}
Here we use the one-to-one relationship between elements of $B_k$ and
elements of $S_{k+}|_{x_3=0}$ to find the best approximation
of a function $f\in L^2(\real^2)$ by a Herglotz wave function in $S_{k+}|_{x_3=0}$.
Furthermore, we show that there is nothing to gain if we consider
Herglotz wave functions supported over all of $S(0,k)$, instead of
only supported on the upper half sphere, that is replacing $S_{k+}$ by
$S_k$.

\begin{theorem}
\label{thm:herglotz}
Let $f \in L^2(\real^2)$. The best approximation
of $f$ by a Herglotz wave function in $S_{k+}$ restricted to the plane
$x_3=0$ has density
\begin{equation}
 g(\vvz) = (2\pi)^{-2} \chi_{B(0,k)}(\vz) \hat f (\vz) \frac{(z_3)_+}{k},
 ~\text{with}~|\vvz| = k.
 \label{eq:density}
\end{equation}
\end{theorem}

\begin{proof}
In operator notation, the orthogonal projection onto $B_k$ is $\cF^{-1}
\cB_k \cF$, where $\cF$ is the Fourier transform operator and $\cB_k$ is
the operator of multiplication by the function $\chi_{B(0,k)}$. That
this is indeed an orthogonal projector can be easily verified, see e.g.
\cite{Slepian:1964:PSW}. Therefore the best approximation of $f$ by a
function in $B_k$ is $\tilde{f} = \cF^{-1} \cB_k \cF f$. By
\cref{lem:herglotz}, $\tilde{f}$ can be identified with the restriction
to the plane $x_3 = 0$ of some Herglotz wave function in $S_{k+}$ with
density being \eqref{eq:density}. It follows from the proof of
\cref{lem:herglotz} that $g \in L^2_W(S(0,k))$. Indeed $\hat f \in
L^2(\real^2)$ implies that $\chi_{B(0,k)} \hat f \in L^2(B(0,k))$.
\end{proof}

A natural question to ask is whether we get a better approximation if we
approximate with all the Herglotz wave functions in $S_k$, instead
of limiting ourselves to $S_{k+}$. The approximation we get in
\cref{thm:herglotz} with functions in $S_{k+}$ is already optimal, as we
show next.

\begin{corollary}
Let $f \in  L^2(\real^2)$. The best approximation
of $f$ by a Herglotz wave function in $S_k$ restricted to the plane
$x_3=0$ has density \eqref{eq:density} on the sphere $|\vvz| = k$.
\end{corollary}
\begin{proof}
We show that the space of functions obtained by restricting $S_k$ to
the plane $x_3=0$ is $B_k$, i.e. identical to that obtained by
restricting $S_{k+}$.
For a $u\in S_k$ with density $g(\vvz)$ we have
\[
\begin{aligned}
u(\vvx) &= \int_{S(0,k)} e^{i \vvx \cdot \vvz} g(\vvz) dS(\vvz) \\
&= \int_{S(0,k), z_3>0} [e^{i\vvx \cdot (\vz, z_3)} g_+(\vvz) + e^{i\vvx
\cdot (\vz, -z_3)} g_-(\vvz)] dS(\vvz),
\end{aligned}
\]
where $g_+(\vvz) = g(\vz, z_3)$ and $g_-(\vvz) = g(\vz, -z_3)$.
Restricting $u$ to the plane $x_3=0$ we get 
\[
\begin{aligned}
u(\vx, 0) &= \int_{S(0,k), z_3>0} e^{i\vx \cdot \vz} [g_+(\vvz) +
g_-(\vvz)] dS(\vvz).
\end{aligned}
\]
Clearly $g_+$ and $g_-$ are in $L^2_W(S(0,k))$, hence we have that
$u(\vx,0) \in B_k$.
\end{proof}

\subsection{Does time reversal solve the approximation problem?}
\label{sec:tr}
In \cref{sec:timerev} we saw that time reversal of a volumetric source
distribution $f$ is equivalent to convolution with
$(4\pi)^{-1}j_0(k|\cdot|)$ and is related to the best approximation in a
ball (\cref{sec:herglotz:vol}). For functions that are supported on the
plane $x_3=0$, we show in \cref{thm:2dtr} that
convolution with $j_0(k|\cdot|)$  is a
filter with spatial frequency response supported on $B(0,k)$, as
the filter we obtained in \cref{sec:bestapprox} but suboptimal.
\begin{theorem}
 \label{thm:2dtr}
 Let $f\in L^2(\real^2)$ be such that
 \[
  \int_{B(0,k)} \frac{|\hat f(\vxi)|^2}{k^2 - |\vxi|^2} d\vxi <
  \infty.
 \]
 Time reversal of
 point sources modulated by $f$ and located on the plane $x_3=0$
 gives the Helmholtz equation solution $u(\vvx) =
 (j_0(k|\cdot|)*f)(\vvx)$, where the convolution is over $x_1,x_2$.
 The restriction of $u(\vvx)$ to the plane
 $x_3 = 0$ has Fourier transform
 \[
  (u(\vx,0))^\wedge = \frac{2\pi}{k \sqrt{k^2 - |\vxi|^2}} \hat
  f(\vxi) \chi_{B(0,k)}(\vxi).
 \]
 Moreover, $u(\vvx)$ is a Herglotz wave function with density 
 \[
  g(\vvz) = (2\pi)^{-1} k^{-2} \hat f (\vz) \chi_{z_3>0}(\vvz),
 \]
on the sphere $|\vvz| = k$. 
\end{theorem}

\begin{proof}
From the Funk-Hecke formula (see e.g. \cite[\S 2.4]{Colton:1998:IAE}), we get
\begin{equation}
 j_0(k|\vvx|) = \frac{1}{4\pi k^2}\int_{S(0,k)} e^{i \vvz
 \cdot \vvx} dS(\vvz),
 \label{eq:funk}
\end{equation}
and thus $j_0(k|\vvx|)$ is a Herglotz wave function with constant
density $1/(4\pi k^2)$ on $S(0,k)$. The restriction $j_0(k|\vx|) = j_0(k|(\vx,0)|)$ of
$j_0(k|\vvx|)$ to the plane $x_3 = 0$ is
\[
 j_0(k|\vx|) =  \frac{1}{4\pi k^2}\int_{S(0,k)} e^{i \vz
 \cdot \vx} dS(\vvz) = \frac{1}{2\pi k^2}\int_{S(0,k),z_3>0} e^{i \vz
 \cdot \vx} dS(\vvz),
\]
where the last equality comes from changing variables in the lower half
sphere. By a reasoning similar to the proof of \cref{lem:herglotz}, the Fourier
transform of $j_0(k|\vx|)$ as a function on $\real^2$ is:
\[
 [j_0(k|\cdot|)]^\wedge(\vxi) = \frac{2 \pi}{k \sqrt{k^2-|\vxi|^2}}
 \chi_{B(0,k)}(\vxi).
\]
In the frequency domain the convolution becomes
\[
 \begin{aligned}
  (j_0(k|\cdot|) * f)^\wedge(\vxi) & =  [j_0(k|\cdot|)]^\wedge(\vxi) \hat
  f(\vxi)\\
  &= \frac{2\pi}{k \sqrt{k^2-|\vxi|^2}}
 \hat{f}(\vxi)\chi_{B(0,k)} (\vxi).
 \end{aligned}
\]
Hence the hypothesis on $f$ guarantees that $j_0 (k|\cdot|) * f \in B_k$. By
\cref{lem:herglotz}, we know that $j_0(k|\cdot|) * f$ is a Herglotz wave
function with bounded density supported on the half sphere $\{\vvz ~|~
\vvz\in S(0,k), z_3\ge 0\}$. Its density can be obtained using
\cref{thm:herglotz}.
\end{proof}

\section{Numerical experiments}
\label{sec:numerics}
We saw in \cref{sec:herglotz:vol} that the best approximation of a
function $f$ by Herglotz wave functions is essentially given by keeping
only the Fourier components of $f$ that lie on the sphere $S(0,k)$. We
illustrate this procedure with numerical experiments in
\cref{sec:numerics:volume}, where as expected we get poor approximations
of the function $f$ in a volume. When the goal is to approximate a
function in the plane by Herglotz wave functions restricted to the same
plane, we saw in \cref{sec:herglotz:plane} that we can expect the
approximation to be a low pass filtered version of the function we wish
to approximate.  We illustrate this procedure and compare it to time
reversal in \cref{sec:numerics:area}.

\subsection{Approximation on a volume}
\label{sec:numerics:volume}
We recall from \cref{sec:herglotz:vol} that the best approximation of a
function by spherical wave functions on a ball of radius $R$ is a
Herglotz wave function with density given by \eqref{eq:frn:density}.
Instead of calculating \eqref{eq:frn:density}, we use the asymptotic
result in \cref{thm:relation} which shows that the density
\eqref{eq:frn:density} approaches $\hat f|_{S(0,k)}$ in some sense.
We approximate the Herglotz wave function with density $\hat
f|_{S(0,k)}$ as follows.
\begin{enumerate}[Step 1.]
 \item Discretize the sphere $S(0,k)$ with a Delaunay triangulation.
 \item  Use a uniform spatial grid  to discretize $f$ on a cube. Using
  the DFT, this gives an approximation to $\hat f$ on a uniform (spatial)
  frequency grid.
 \item Approximate $\hat f$ at the triangle centers of
  the triangulation of $S(0,k)$ by linearly interpolating the
  $\hat f$ calculated in Step 2.
 \item The Herglotz wave function $u$ with density $\hat{f}$ is given by
  the integral \eqref{eq:hwf} with $g \equiv \hat f$, that we approximate by
  assuming it is piecewise constant on the triangles of the
  triangulation of $S(0,k)$. Thus $u$ is 
  approximated by a finite sum of plane waves with wavenumber $k$.
\end{enumerate}
We illustrate this procedure in \cref{fig:tetrahedron}, with a function
$f$ that is related to
a tetrahedron $T$ with center of mass at the origin and circumscribing
sphere of radius $5\lambda$ and is given by
\[
 f(\vvx) =  |T| \chi_{B(0,5\lambda)}(\vvx) - |B(0,5\lambda)|
 \chi_T(\vvx),
\]
where $|A|$ denotes the volume of some region $A$. We chose this $f$
because its nodal set is $\partial T$, the boundary of the tetrahedron,
and because $\int f(\vvx) d\vvx = 0$. This last property ensures that
$\hat f(\vvzero) = 0$, as is the case for all Herglotz wave functions.
The function $f$ is sampled on a uniform grid with $256^3$ points in the
cube $[-5\lambda,5\lambda]^3$. This corresponds to a uniform grid in
spatial frequency within the cube $[-12.8k,12.8k]^3$ with identical number
of points. The triangulation of $S(0,k)$ that we used consisted of 7292
triangles and was obtained using the DistMesh package
\cite{Persson:2004:SMG}. The visualization of the resulting Herglotz
wave function $u$ is done in slices $(x,y) \in [-5\lambda,5\lambda]^2$
that are sampled with $100^2$ uniformly spaced points.

Since the asymptotic in \cref{thm:relation} is not a projection, the
resulting Herglotz wave function will be off by a scaling factor, which
is why we include a comparison of the zero level sets in
\cref{fig:tetrahedron}. This numerical experiment illustrates the poor
approximation that was expected in a volume. Indeed this function has
significant Fourier components outside of $S(0,k)$ that are filtered out
in $u$. Intuitively, we are trying to control a volume with only a two
dimensional field (the Herglotz wave function density).

\begin{figure}
\centering
\begin{tabular}{ccc}
	\includegraphics[width = 0.30\linewidth]{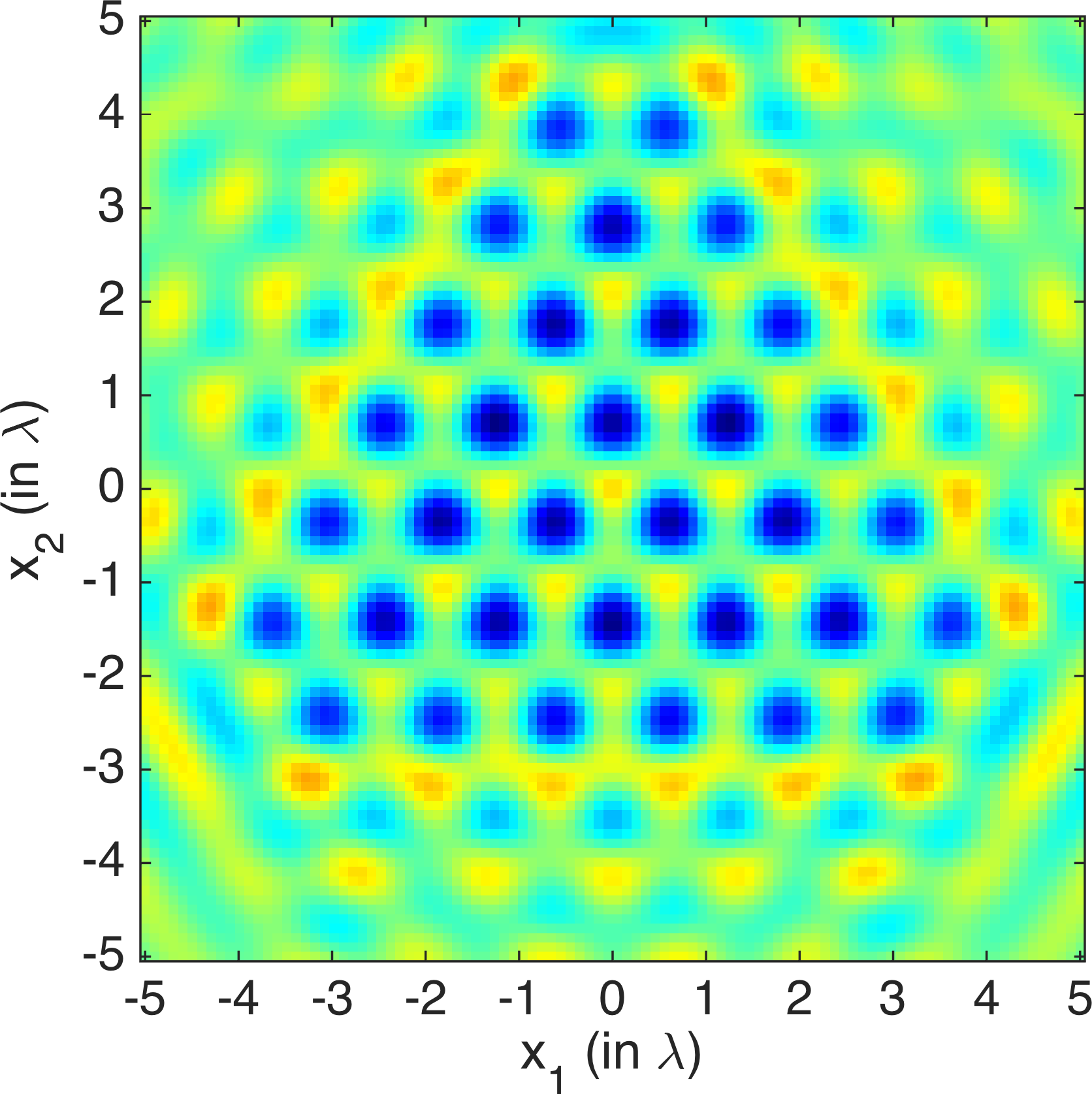}
&
	\includegraphics[width = 0.30\linewidth]{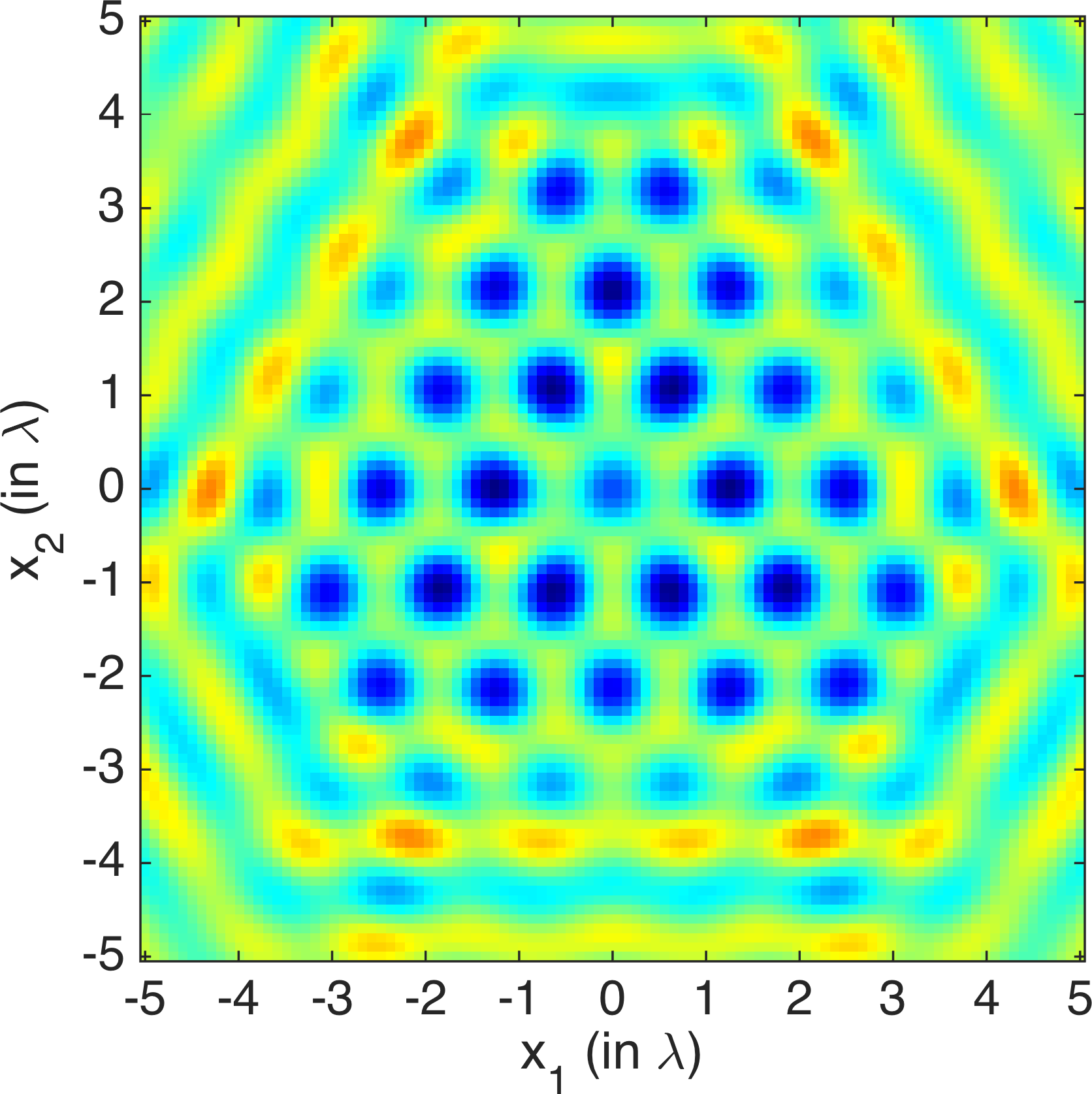}
&
	\includegraphics[width = 0.30\linewidth]{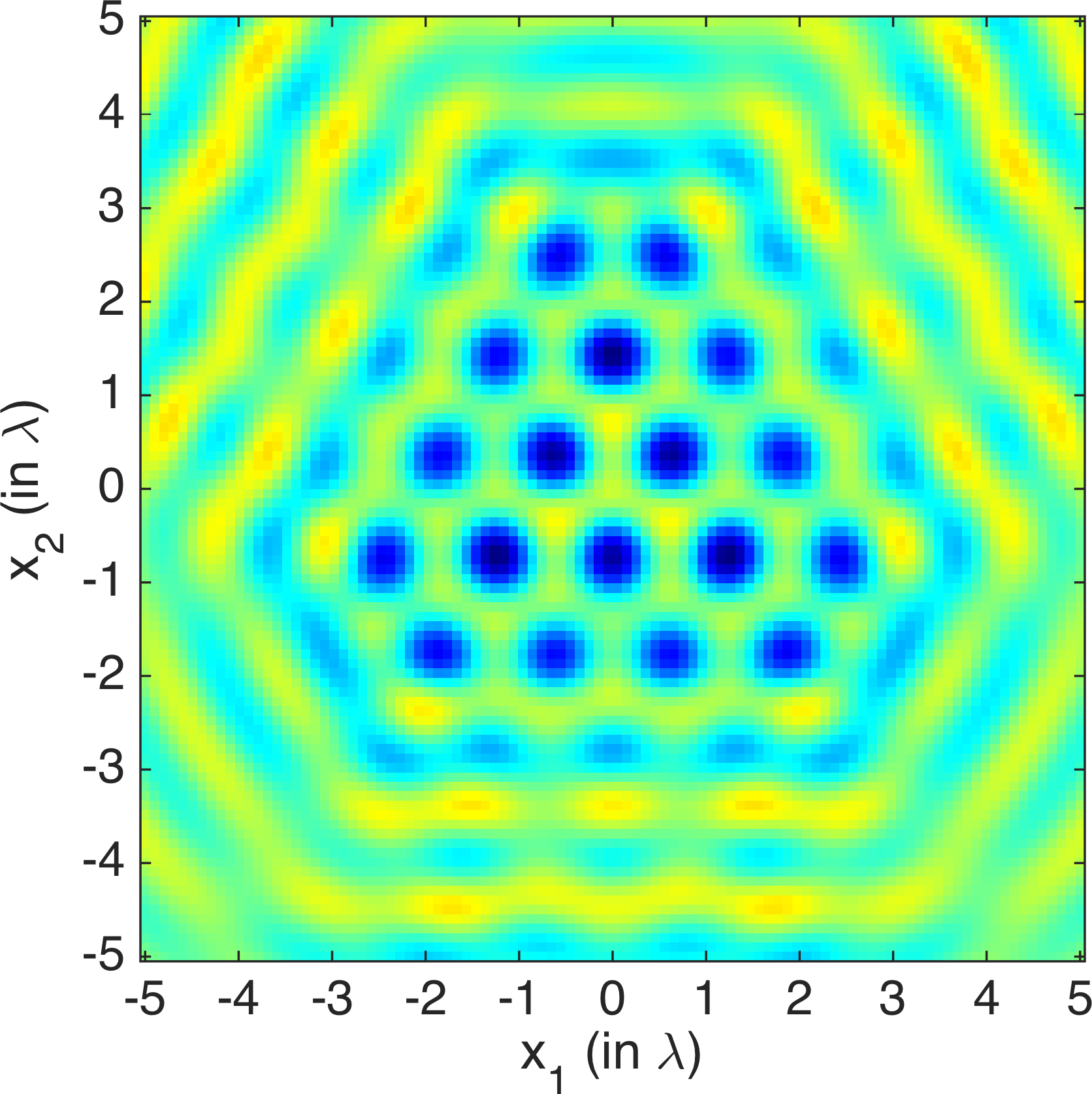}
\\
	\includegraphics[width =0.30\linewidth]{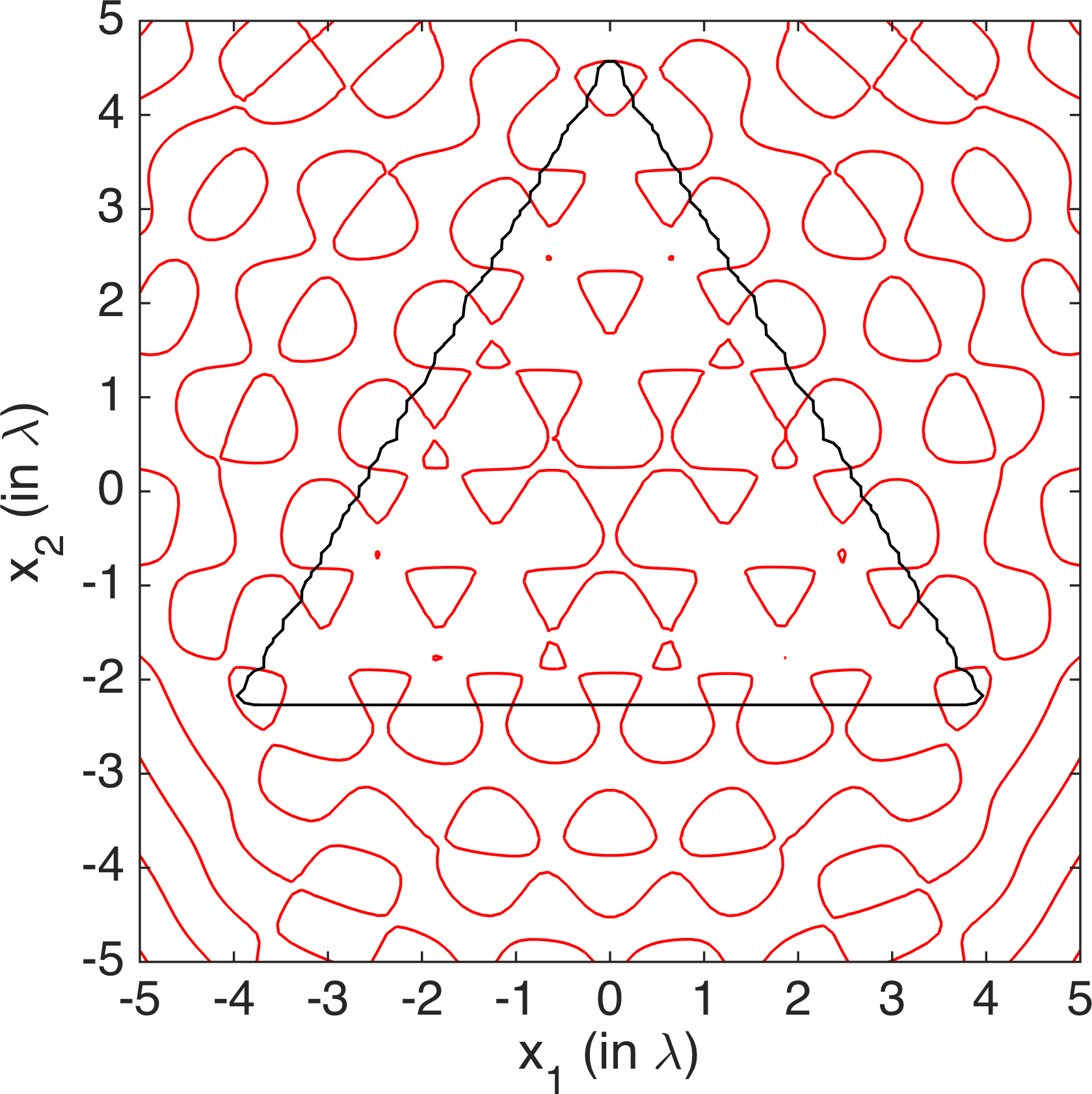}
&
	\includegraphics[width = 0.30\linewidth]{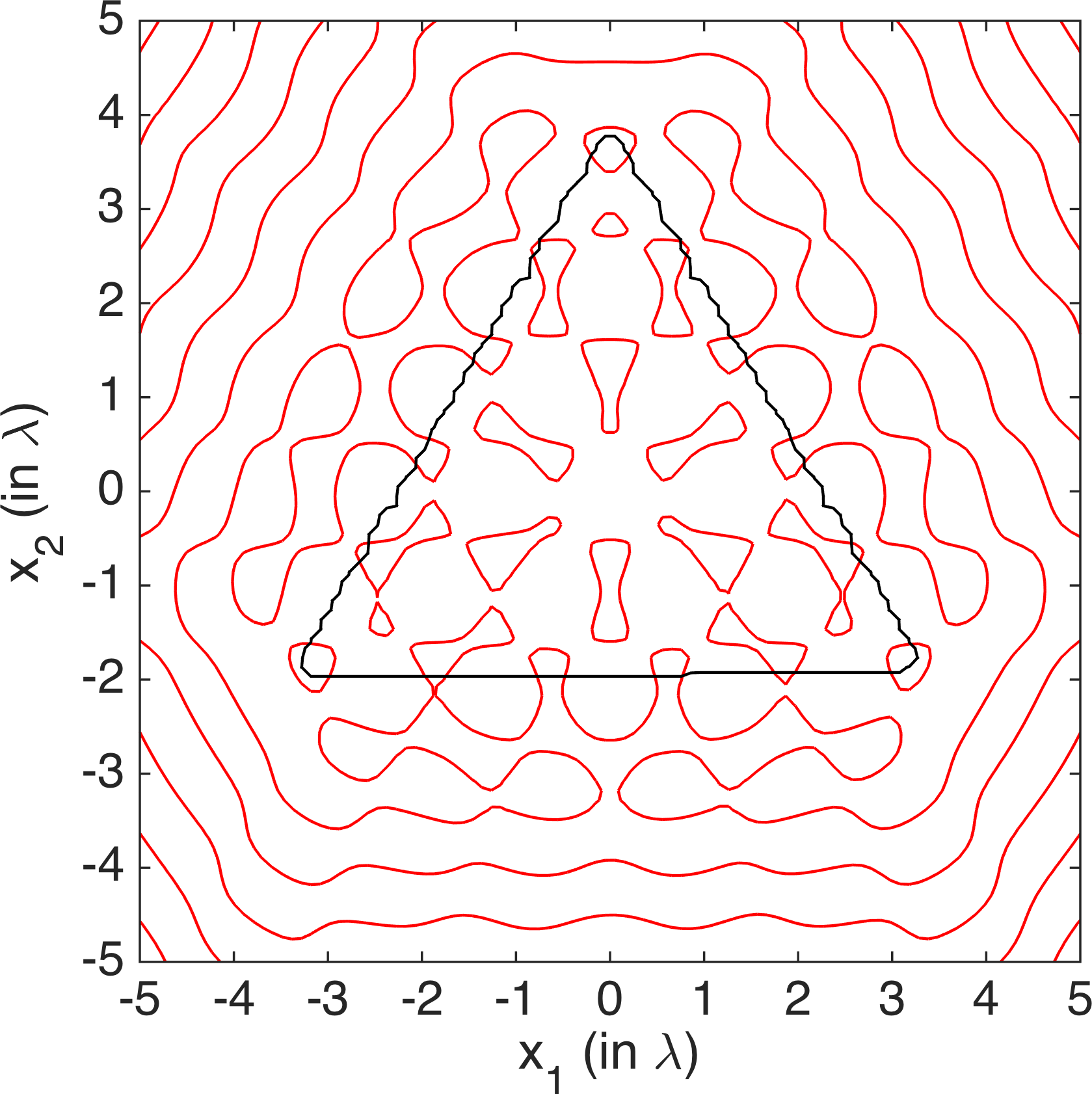}
&
	\includegraphics[width = 0.30\linewidth]{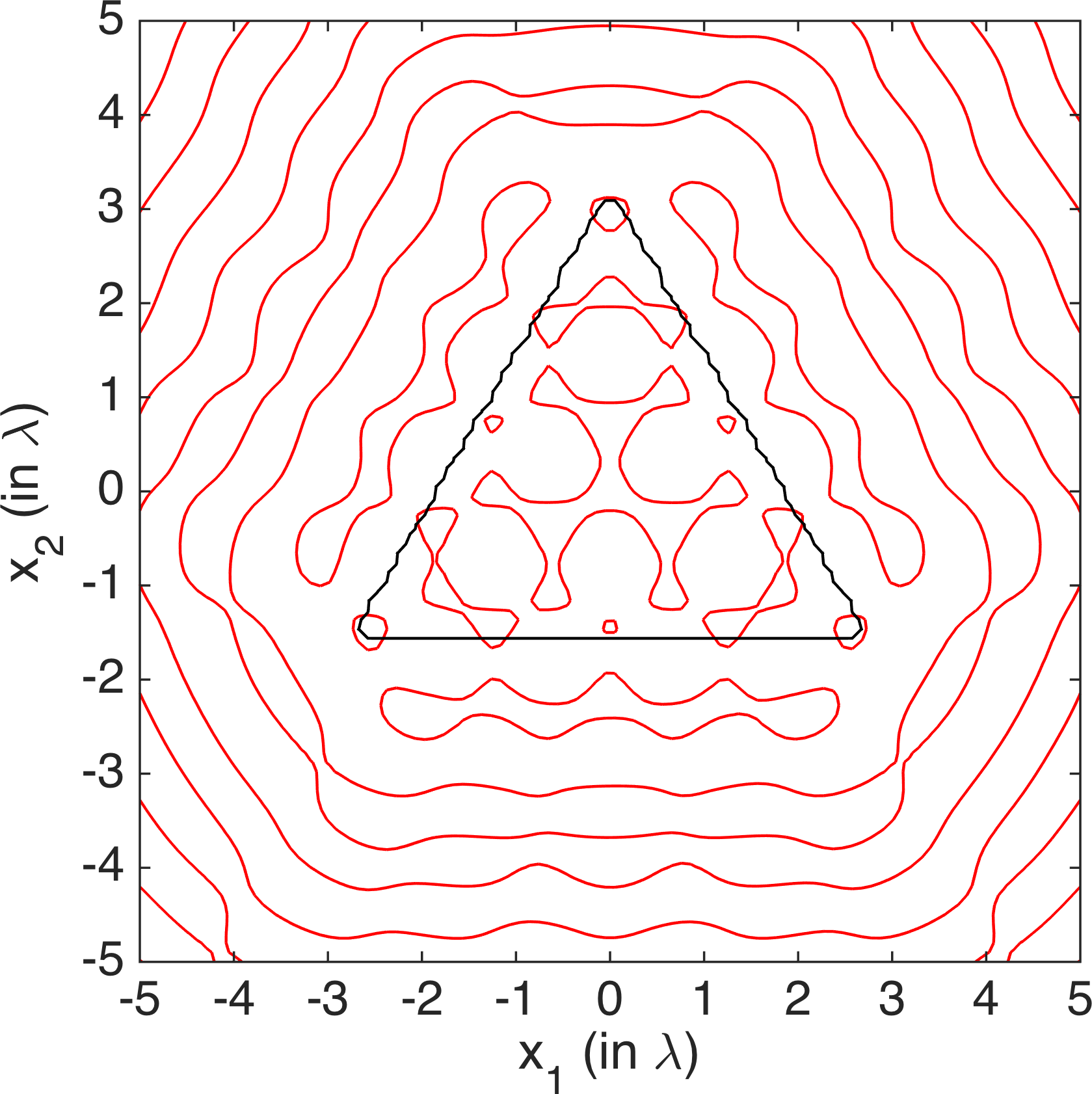}
\\
$x_3=-\lambda$
&
$x_3=0$
&
$x_3=\lambda$
\end{tabular}
\caption{Approximation of a function $f$ with tetrahedral zero level
 set. The top row shows three $x_3$ constant slices 
of the real part of the approximation of $f$. The bottom row shows zero level
sets of the original function $f$ (black) and the real part of the approximation 
(red) for each slice. The color scale in the top row ranges from -0.01 (blue) to 
0.01 (red).}
\label{fig:tetrahedron}
\end{figure}

\subsection{Approximation on a plane}
\label{sec:numerics:area}
We compare two methods for approximating a compactly supported
bounded function $f$  by Helmholtz equation solutions: time reversal,
and projection onto the space of Herglotz wave functions supported on
the upper half of the sphere $S(0,k)$ and restricted to the plane $x_3=0$.
As we saw in \cref{thm:herglotz}, the projection method is optimal
and should give better approximations. This is confirmed in the
numerical results we present here.

\subsubsection{Time reversal approximation}
As explained in \cref{sec:timerev}, we can view time reversal as a
convolution with kernel $(4\pi)^{-1}j_0(k| \cdot |)$, which is the
tightest an entire solution to the Helmholtz equation can get. The function $f$
that we use is obtained from the
University of Utah ``U" logo by rescaling it so that it is about
$6\lambda \times 6\lambda$, and setting $f(x) = 1$ inside the ``U'' and
$f(x)=-1$ outside.  We use a uniform grid on the square
$[-5\lambda,5\lambda]^2$ with $512^2$ points.  

The time reversal
approximation is shown in \cref{fig:convolution}. 
\Cref{thm:2dtr} shows that the time reversal procedure is equivalent, up
a to a scaling factor, to filtering the function we wish to approximate
with a filter response function that is supported in $B(0,k)$ but that
boosts the frequencies with $|\vxi|$ close to $k$.  We report only results
up to a proportionality constant, and evaluate the performance of this
method by comparing zero level sets.  This procedure gives a Herglotz
wave function approximation to $f$ that is suboptimal (see
\cref{sec:tr}) compared to the projection method we see next.
\begin{figure}
\begin{center}
\begin{tabular}{cc}
	\includegraphics[width = 0.45\linewidth]{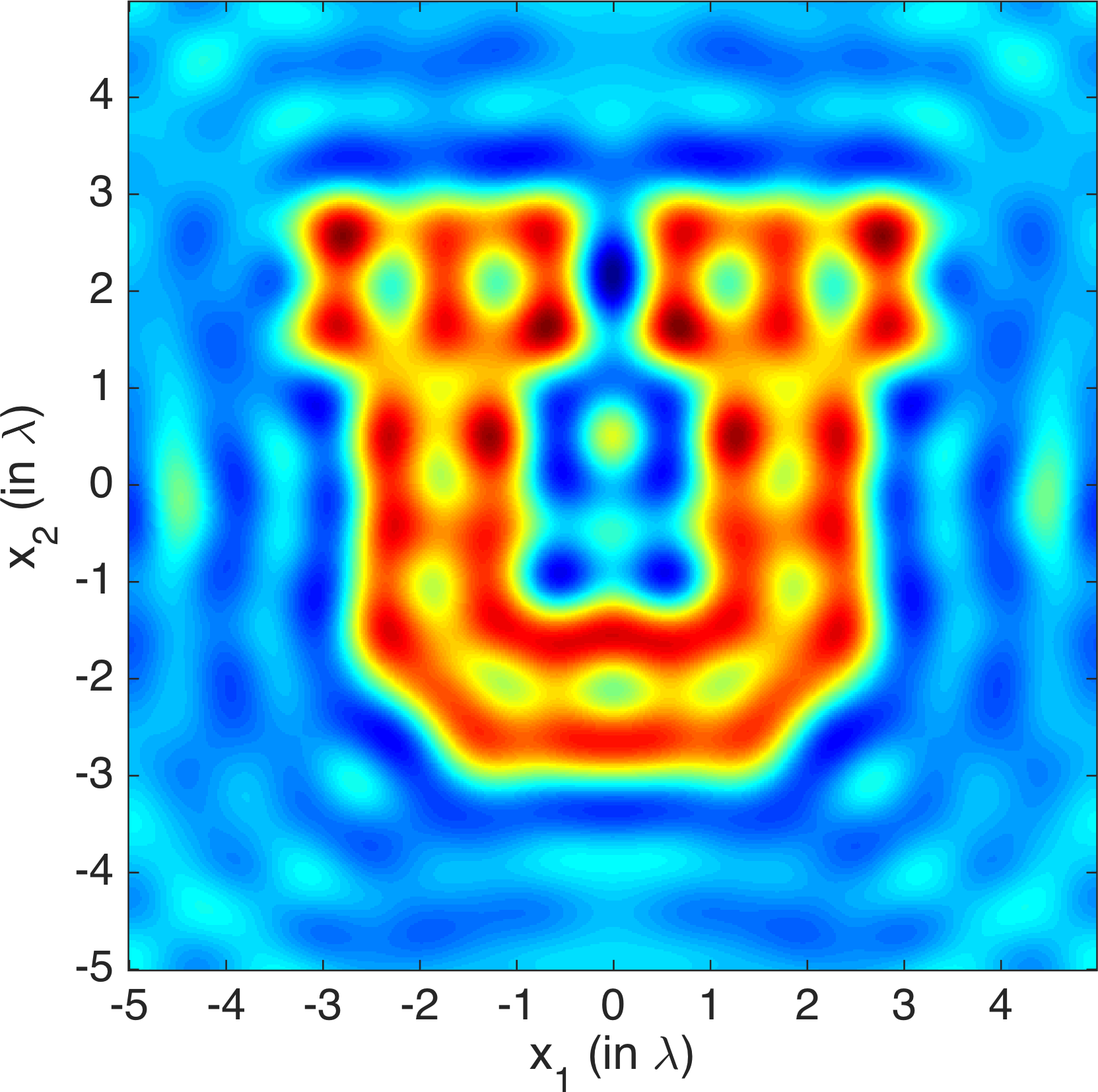}
&
	\includegraphics[width = 0.45\linewidth]{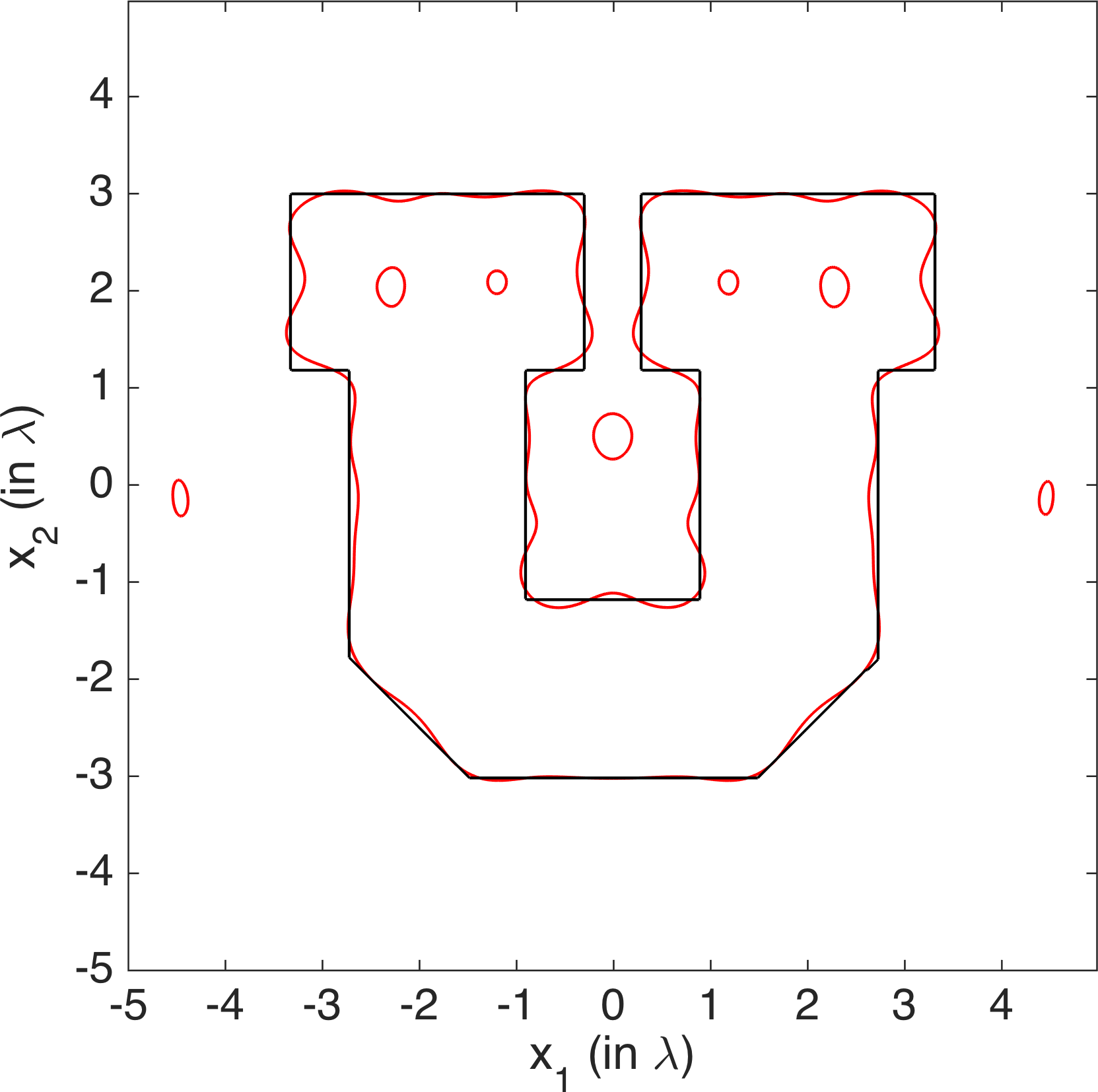}
\\
(a)
&
(b)
\end{tabular}
\end{center}
\caption{Approximating a function $f$ with ``U'' zero level set using
 time reversal. (a) shows the convolution of $f$ with $j_0(k|\cdot|)$.
 (b) shows zero level sets of $f$ (black) and of the convolution (red).}
\label{fig:convolution}
\end{figure}

\subsubsection{Best approximation with Herglotz wave functions}

Following \cref{sec:herglotz:plane}, the best approximation by Herglotz
wave functions of some function $f$ can be obtained by first
filtering the function $f$, eliminating all spatial frequencies $\vxi$
with $|\vxi| > k$ and then using \cref{lem:herglotz} to relate the
filtered $f$ to a Herglotz density on $S(0,k)$.

To filter the function $f$ (i.e. projecting it onto $B_k$) we proceed as
follows.
\begin{enumerate}[Step 1.]
 \item Use a uniform spatial grid in a square to discretize $f$. By using the DFT,
  this gives an approximation to $\hat f$ on a uniform (spatial)
  frequency grid. 
\item Restrict the $\hat f$ computed in Step 1 to the ball $B(0,k)$.
\item Calculate the inverse Fourier transform of $\hat f|_{B(0,k)}$
 (using the DFT).
\end{enumerate}
The effect of filtering is shown in \cref{fig:projection}, for the same
function $f$ with University of Utah logo. As in the time reversal
experiments we discretized $f$ on a uniform grid of the square
$[-5\lambda, 5\lambda]^2$ with $512^2$ points. This corresponds to a
spatial frequency grid on the square $[-25.6k,25.6k]^2$ with the same
number of points. Since the frequencies $|\vxi| > k$ have been
eliminated, the corners in the zero level set of $f$ are smoothed out in
the filtered version.

\begin{figure}
\begin{center}
\begin{tabular}{cc}
	\includegraphics[width = 0.45\linewidth]{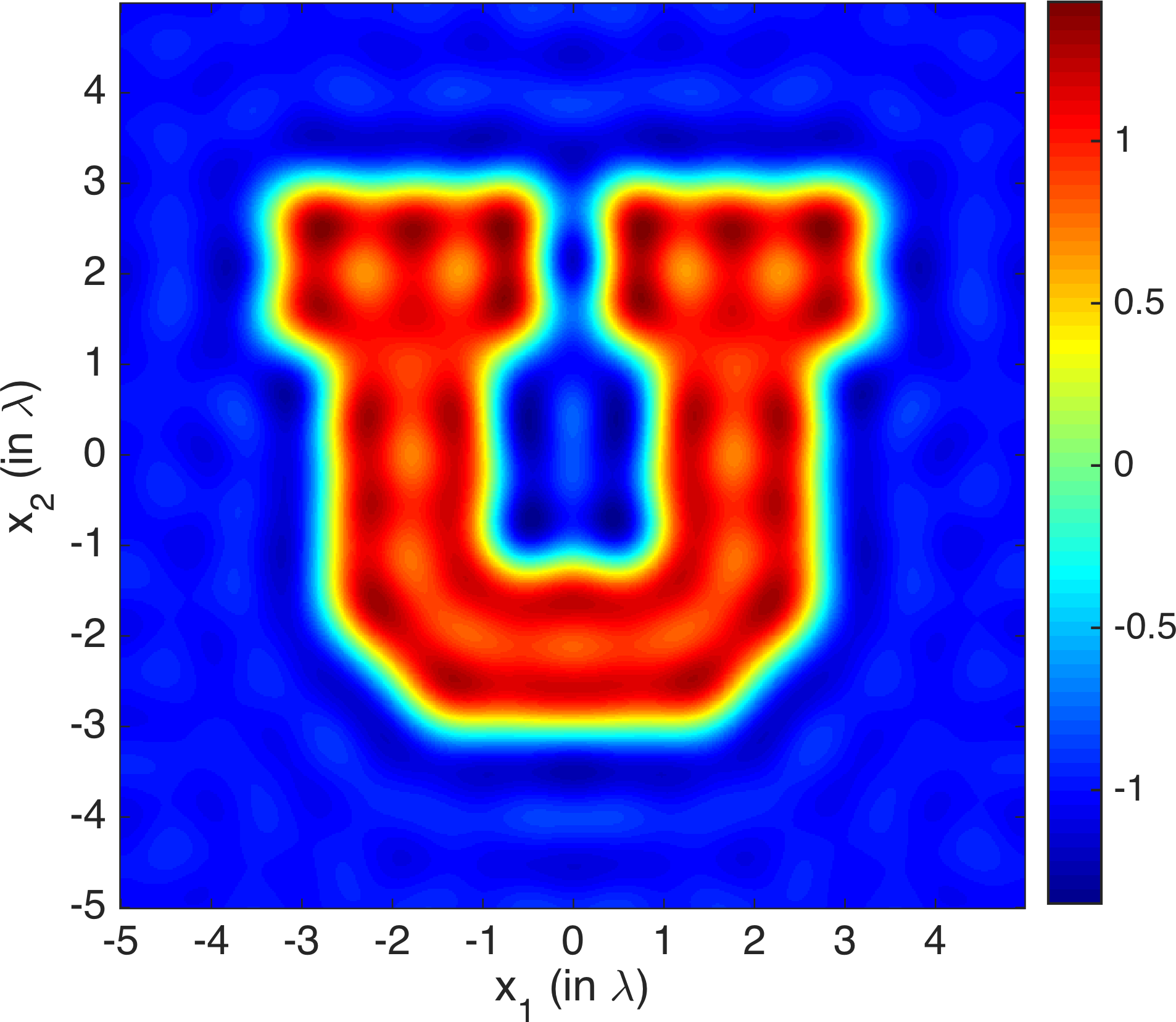}
&
	\includegraphics[width = 0.39\linewidth]{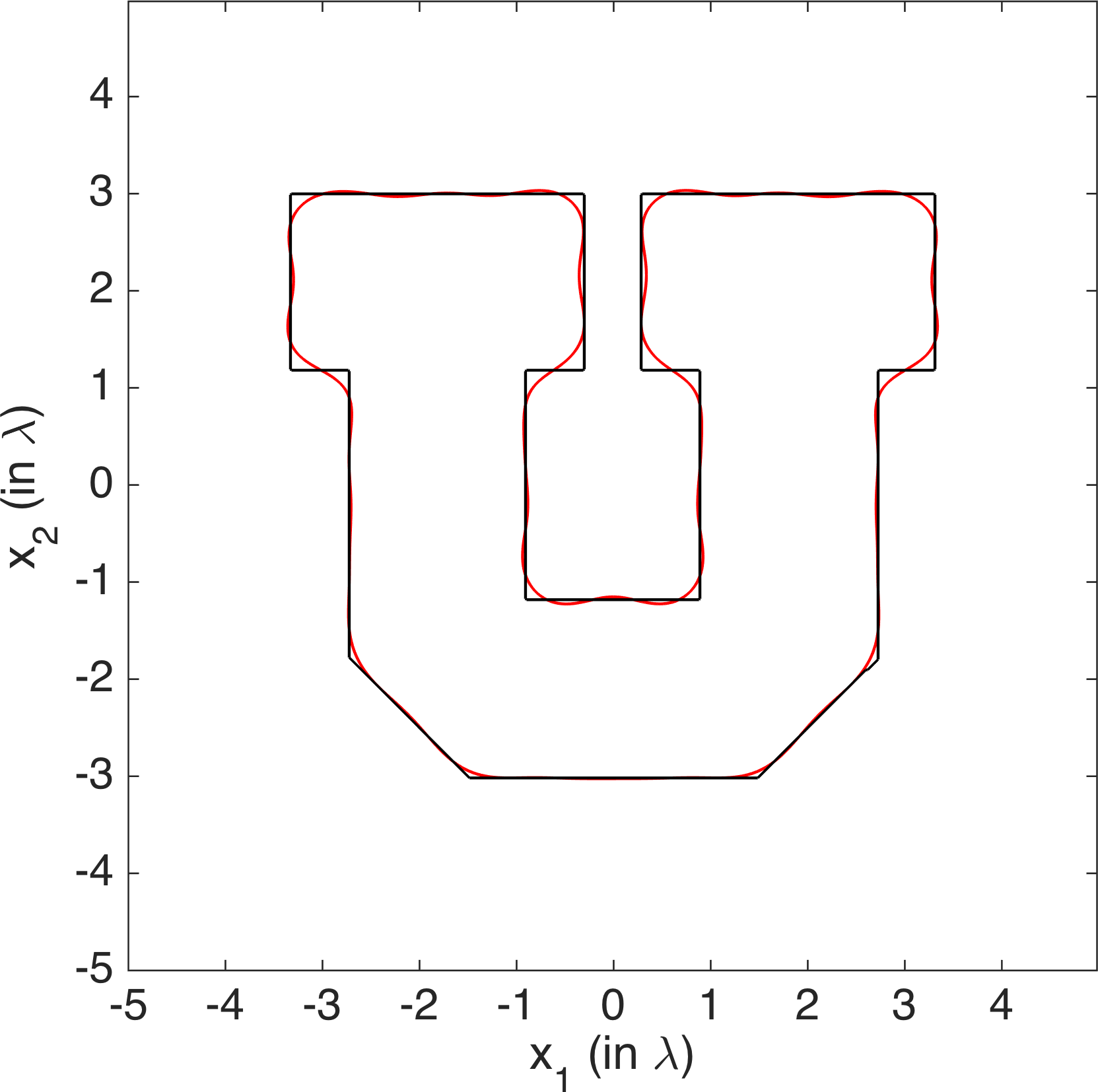}
\\
(a)
&
(b)
\end{tabular}
\end{center}
\caption{Approximating a function with ``U" zero level set using the projection
approach described in \cref{sec:numerics:area}. (a) Real part of
projection of $f$ onto $S_{k+}|_{x_3=0}$. (b) Zero level set of real
part of projection.}
\label{fig:projection}
\end{figure}

We recall that \cref{lem:herglotz} gives a one-to-one
relationship between functions in $B_k$ and functions in 
$S_{k+}|_{x_3=0}$. In other words, the best approximation
of $f$, which we obtained by projecting $f$ onto the space
of band-limited functions $B_k$, corresponds to the restriction
of a Herglotz wave function on the upper half of the sphere 
$S(0,k)$ to the plane $x_3=0$. The Herglotz density we used in the
numerical experiment appears in \cref{fig:mollweide}, where we can clearly see it
is supported on the upper half sphere. We use this relation to
evaluate a Herglotz wave function with density
\cref{eq:density} on the plane $x_3=0$. This is done as follows.
\begin{enumerate}[Step 1.]
\item Discretize the sphere $S(0,k)$ with a Delaunay triangulation.
\item Use a uniform spatial grid to discretize
$f$. With the DFT these give an approximation of $\hat f$ on a uniform
  (spatial) frequency grid.
\item Interpolate $\hat f$ on a uniform grid in the $x_3=0$ plane.
\item The Herglotz wave function $u$ with density
\cref{eq:density} is given by the integral \cref{eq:hwf} with
$g$ given by \cref{eq:density}, that we approximate by assuming
it is piecewise constant on the triangles of the triangulation 
of $S(0,k)$. 
\end{enumerate}

For the particular numerical experiment we present, the triangulation of
$S(0,k)$ consisted of 7292 triangles and was obtained using the DistMesh
package \cite{Persson:2004:SMG}. The function $f$ was sampled on a
uniform grid with $512^2$ points on the square $[-5\lambda,
5\lambda]^2$, which corresponds to a uniform grid in spatial frequency
in the square $[-25.6k, 25.6k]^2$ with an identical number of points.
As illustrated in \cref{fig:herglotz}, the Herglotz wave function
restricted to the $x_3=0$ plane is very close to the projection of $f$
onto $B_k$. Any differences are due to discretization errors, as they
should be identical by \cref{lem:herglotz}. To emphasize that the
approximation is a wave field, we also display in \cref{fig:herglotz} the field
at the planes $x_3 = \pm\lambda$. As expected, the approximation
degrades as we move away from the plane $x_3=0$. The $x_3=\pm \lambda$
slices
look identical because the Fourier transform of the function we
considered has a relatively small imaginary part. Nevertheless the
relative difference between the fields at $x_3=\lambda$ and
$x_3=-\lambda$ is about 65\%, when both real and imaginary parts are
kept.

\begin{figure}
\begin{center}
\begin{tabular}{cc}
	\includegraphics[width = 0.45\linewidth]{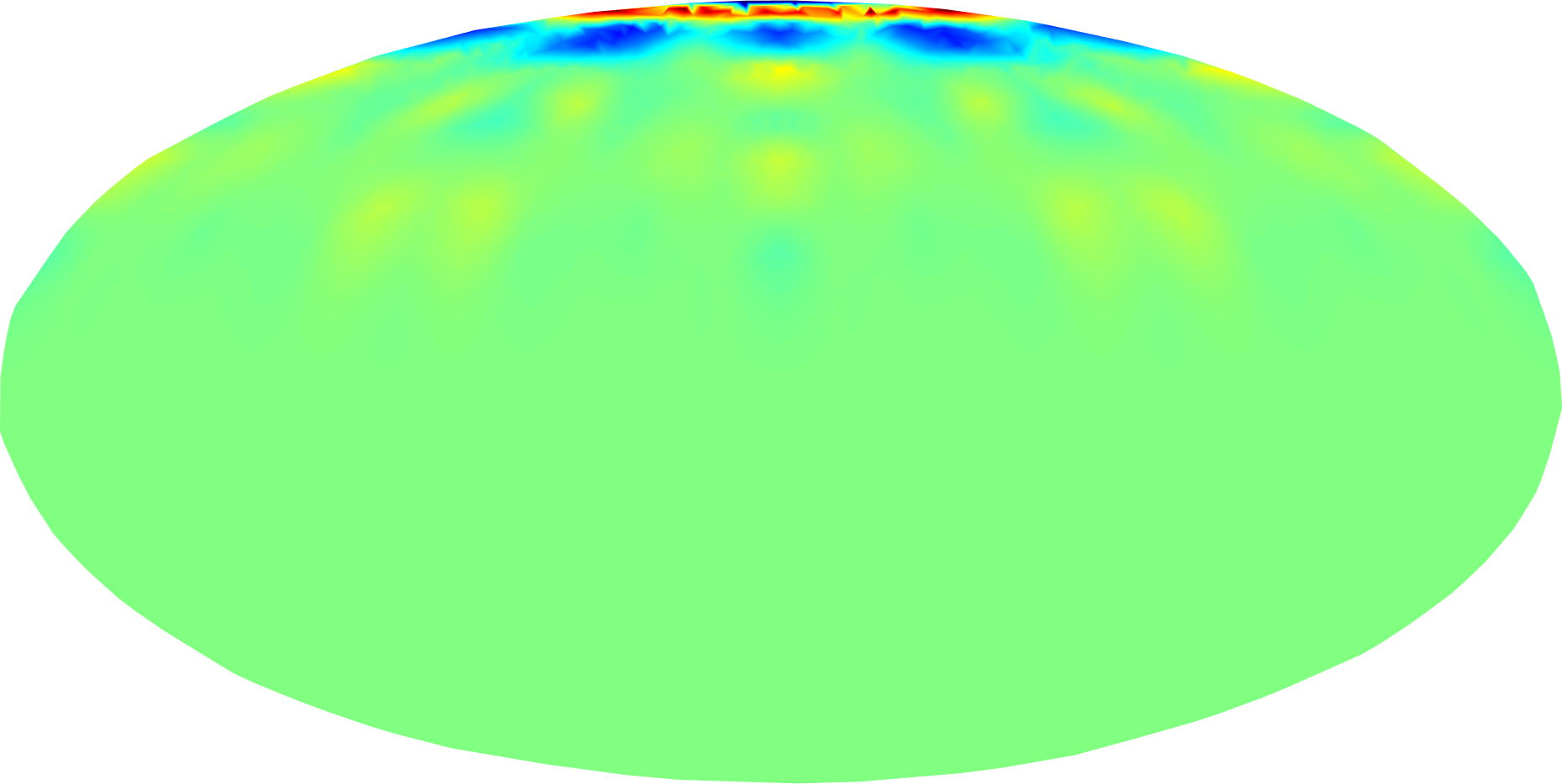}
&
	\includegraphics[width = 0.45\linewidth]{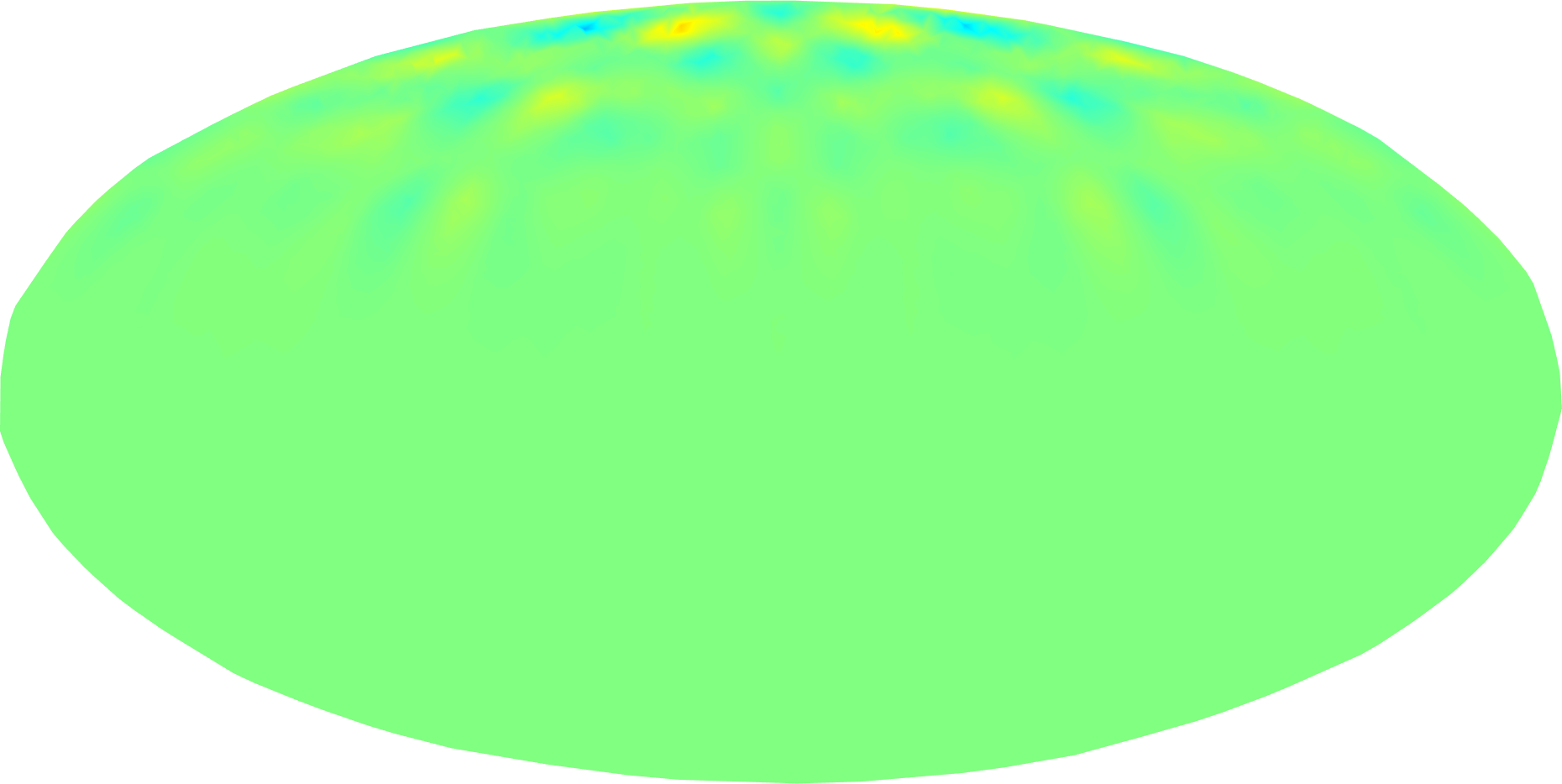}
\end{tabular}
\end{center}
\caption{Real (left) and imaginary (right) parts of the Mollweide projection
for the Herglotz wave function density used to obtain the approximation in
 \cref{fig:herglotz}. The color scale ranges from -20 (blue) to 20
 (red).}
\label{fig:mollweide}
\end{figure}

\begin{figure}
\begin{tabular}{ccc}
	\includegraphics[width = 0.30\linewidth]{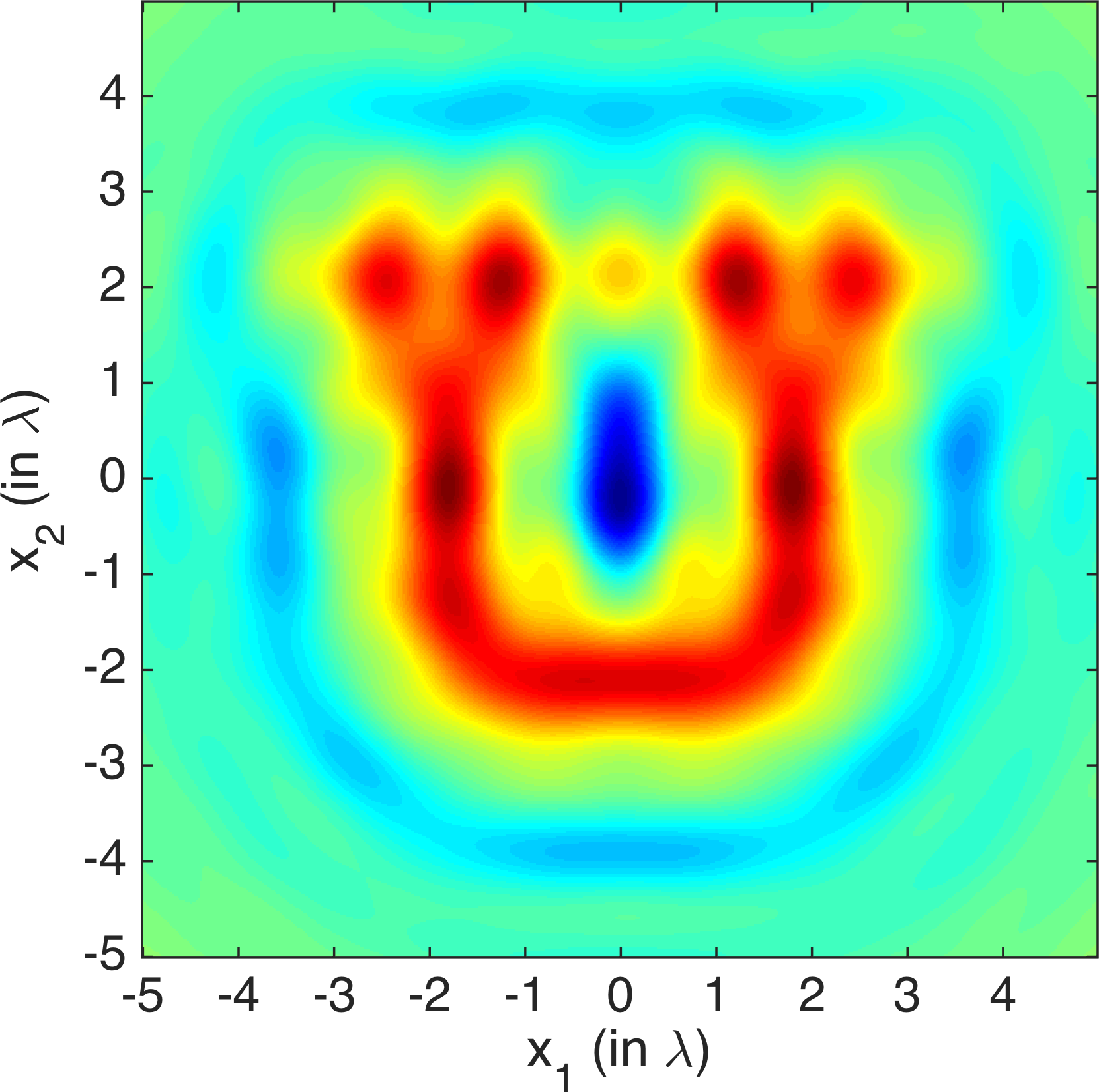}
&
	\includegraphics[width = 0.30\linewidth]{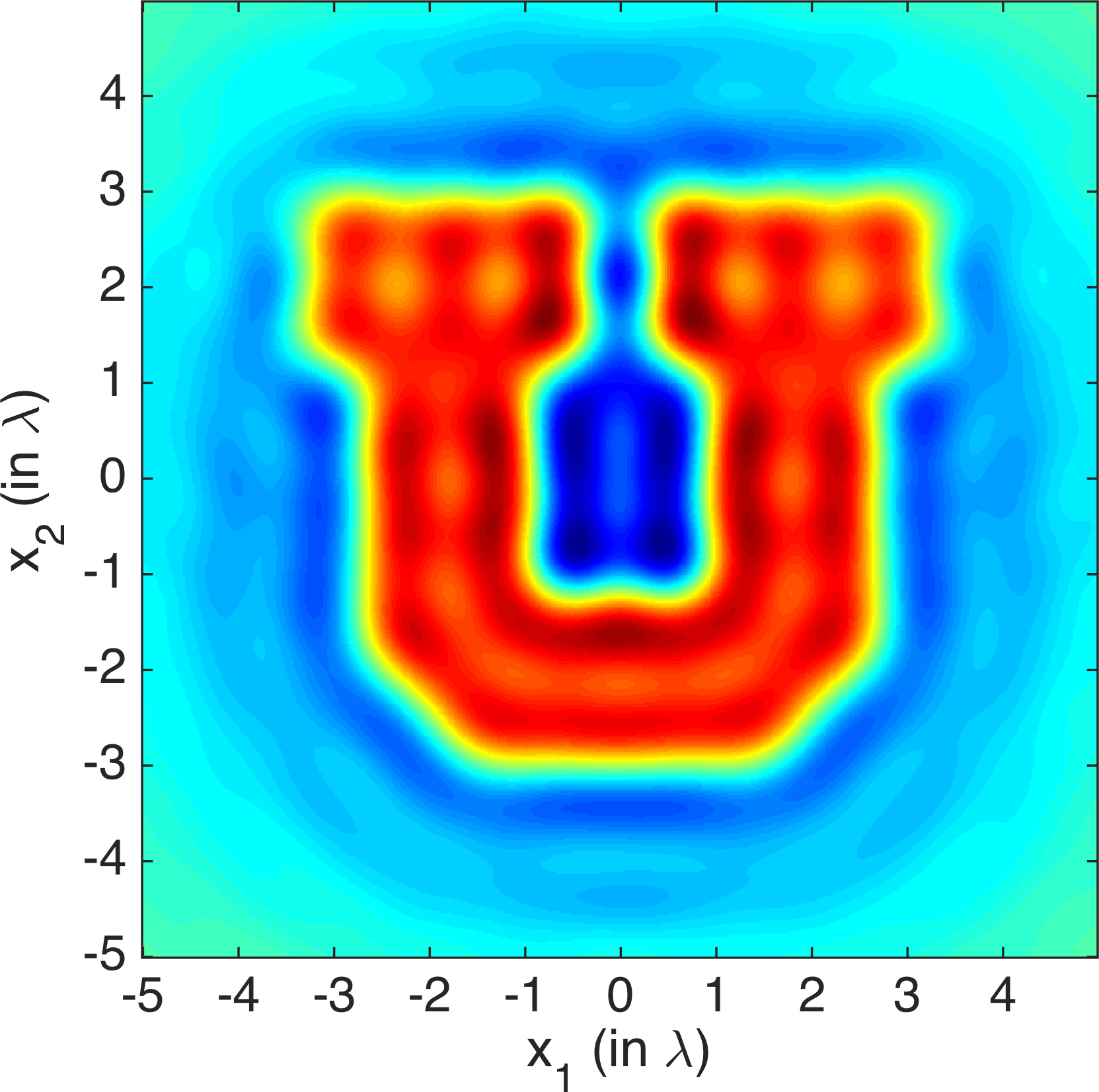}
&
	\includegraphics[width = 0.30\linewidth]{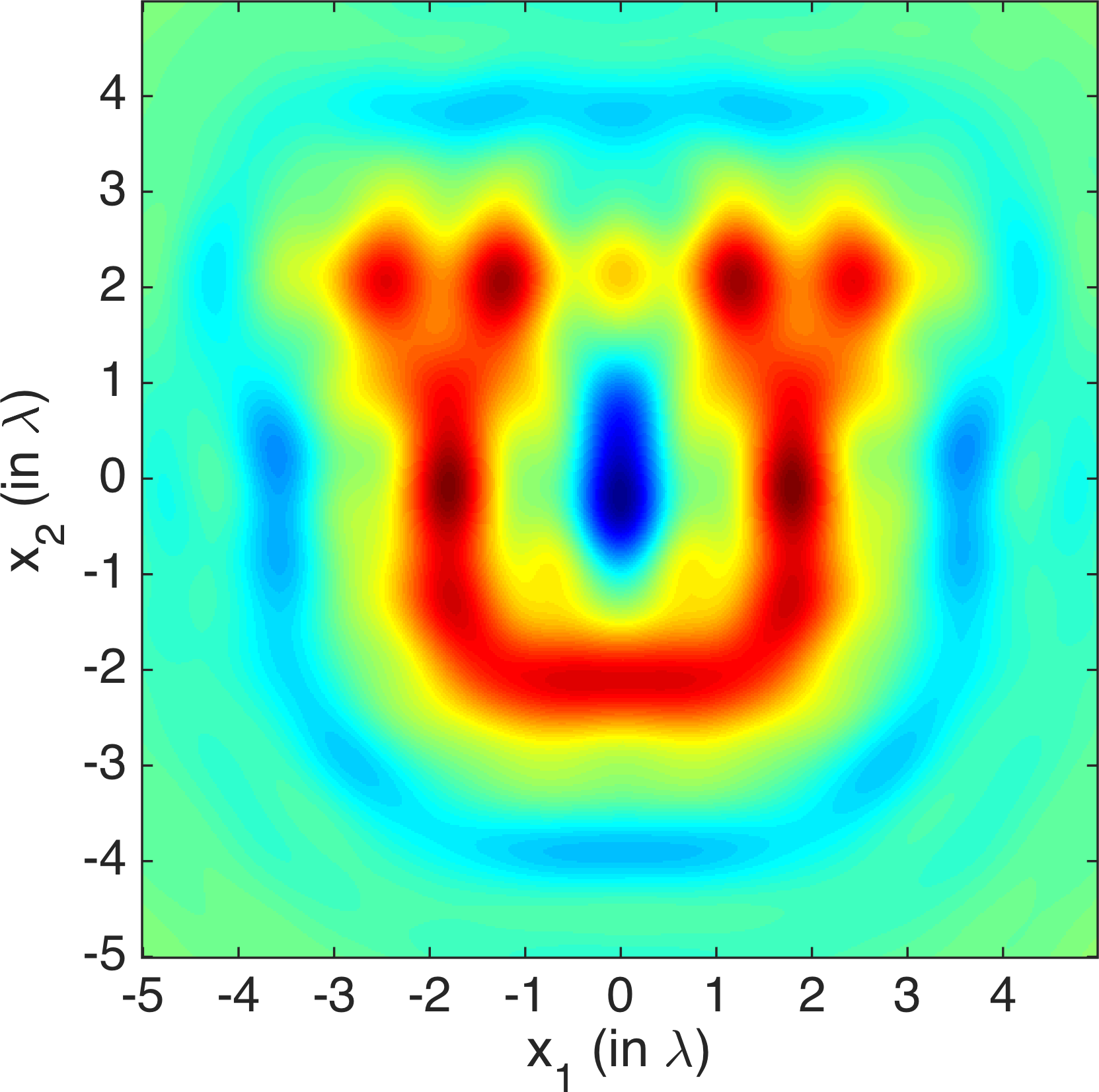}
\\
	\includegraphics[width = 0.30\linewidth]{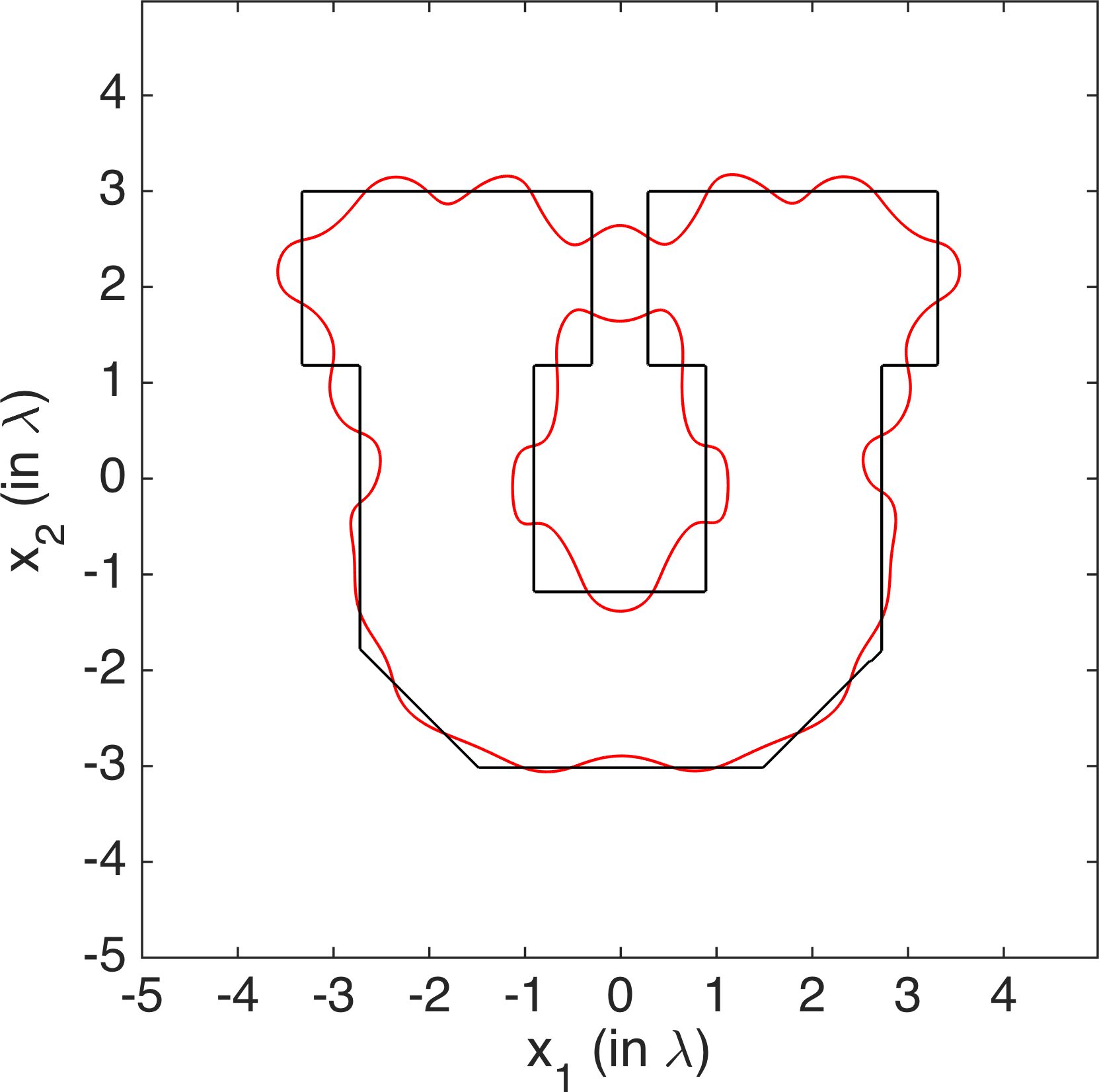}
&
	\includegraphics[width = 0.30\linewidth]{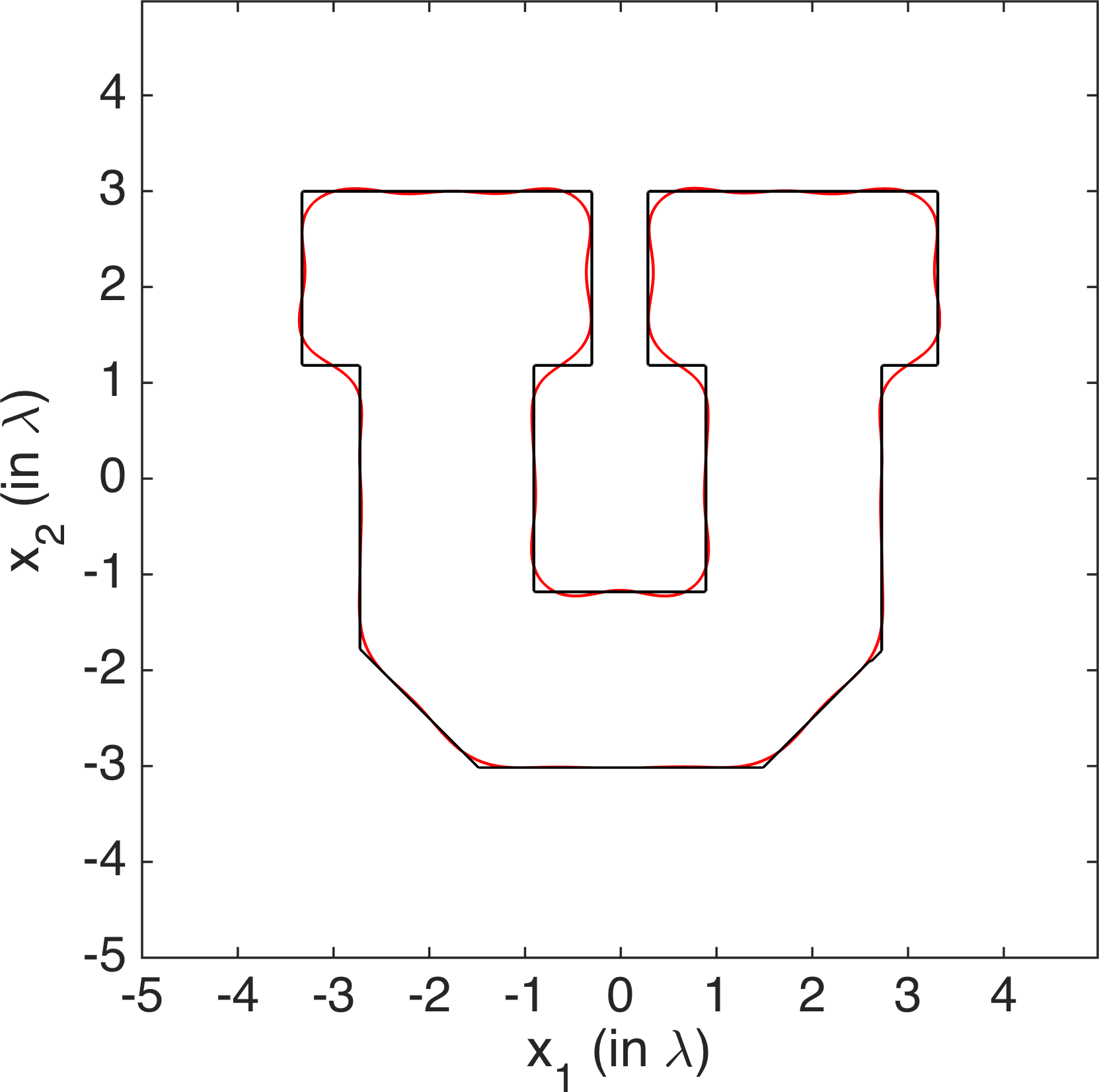}
&
	\includegraphics[width = 0.30\linewidth]{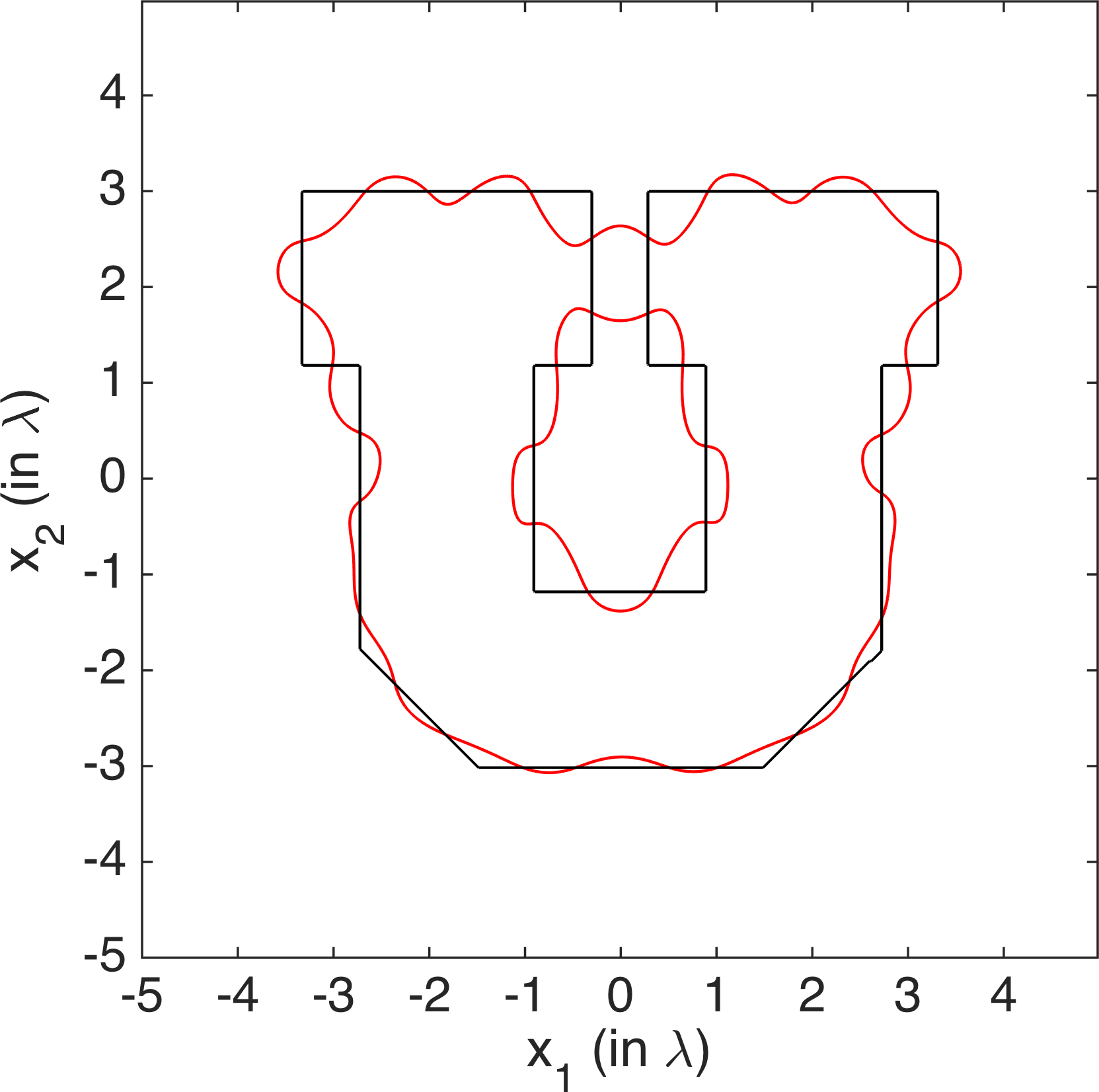}
\\
$x_3=-\lambda$
&
$x_3 = 0$
&
$x_3 = \lambda$
\end{tabular}
\caption{Best approximation of $f$ in $S_{k+}|_{x_3=0}$. 
The top row shows three slices of the real part of the 
approximation of $f$. The bottom row shows zero level sets of the 
original function $f$ (black) and the real part of the approximation 
(red). The color scale in the top row ranges from -2 (blue) to 1.5 
(red). 
}
\label{fig:herglotz}
\end{figure}

\section{Summary and future work}
\label{sec:future}
We have shown that the problem of approximating a function by 3D
Herglotz wave functions is not well-posed if we measure the misfit in a
ball, but becomes well-posed if the misfit is measured in a plane. The
solution to the approximation problem on a ball is asymptotically close
to a time reversal experiment where the function to be approximated is
regarded as a source density. The approximation problem in a plane is
shown to be related to filtering the spatial frequencies of the function
we wish to approximate. Our theoretical results are illustrated by numerical
experiments showing that the approximation problem on a plane gives
Helmholtz equation solutions with nodal set close to that of the
function we approximate. We are currently studying other ways of
measuring misfit to be able to design solutions to the Helmholtz
equation with e.g., maxima along a given curve. This could be used to
manipulate particles in a fluid with acoustic waves, in situations where
the particles cluster about the anti-nodes of a wave.

\section*{Acknowledgements}
 FGV would like to thank Bart Raeymaekers and John
Greenhall for insightful conversations on this topic.

\bibliographystyle{abbrv}
\bibliography{herglotz_bib}

\end{document}